\documentclass{amsart}

\usepackage[utf8]{inputenc}
\usepackage{amsthm, amssymb}
\usepackage[colorinlistoftodos]{todonotes}
\usepackage[colorlinks=true, pdfstartview=FitV, linkcolor=blue, citecolor=purple, urlcolor=purple]{hyperref}
\usepackage{pgf}
\usepackage{tikz}
\usepackage{tikz-cd}
\usetikzlibrary{positioning,shapes,shadows,arrows}
\usepackage[numbers]{natbib}
\usepackage{ytableau}

\newtheorem{prop}{Proposition}
\newtheorem{prob}{Problem}
\newtheorem{thm}[prop]{Theorem}

\newtheorem{lemma}[prop]{Lemma}
\newtheorem{cor}[prop]{Corollary}

\theoremstyle{remark}
\newtheorem{exa}[prop]{Example}
\newtheorem{rem}[prop]{Remark}
\newtheorem{defn}[prop]{Definition}
\numberwithin{prop}{section}
\numberwithin{equation}{section}

\newcommand{\df}{:=}

\newcommand{\fL}{\mathfrak{L}}
\newcommand{\fLb}{\fL^{(\beta)}}
\newcommand{\ofLb}{\overline{\fL}^{(\beta)}}

\newcommand{\Gb}{G^{(\beta)}}

\newcommand{\key}{\mathrm{key}}
\newcommand{\la}{\lambda}

\newcommand{\pab}{\partial^{(\beta)}}
\newcommand{\pib}{\pi^{(\beta)}}
\newcommand{\opib}{\overline{\pi}^{(\beta)}}

\newcommand{\SSYT}{\mathrm{SSYT}}
\newcommand{\SVT}{\mathrm{SVT}}

\newcommand{\word}{\mathrm{word}}
\newcommand{\wt}{\mathrm{wt}}
\newcommand{\ex}{\mathrm{ex}}
\newcommand{\Z}{\mathbb{Z}}

\newcommand{\F}{\mathcal{F}}
\newcommand{\cS}{\mathcal{S}}
\newcommand{\B}{\mathcal{B}}

\newcommand{\emt}{\mathbf{0}}
\definecolor{myred}{RGB}{240, 0, 0}
\definecolor{mypurple}{RGB}{120, 0, 120}
\definecolor{myblue}{RGB}{0, 0, 240}

\author{Tianyi Yu}
\address{Department of Mathematics\\
University of California, San Diego\\
La Jolla, CA, 92093, USA}
\email{tiy059@ucsd.edu}

\title{Set-valued tableaux rule 
for Lascoux polynomials}
\date{September 2021}

\begin{document}

\begin{abstract}
Lascoux polynomials generalize Grassmannian stable Grothendieck polynomials
and may be viewed as K-theoretic 
analogs of key polynomials.
The latter two polynomials have combinatorial formulas involving tableaux:
Lascoux and Sch\"{u}tzenberger gave a combinatorial formula for
key polynomials using right keys;
Buch gave a set-valued tableau formula for Grassmannian stable Grothendieck polynomials. 
We establish a novel combinatorial description for Lascoux polynomials
involving right keys and set-valued tableaux.
Our description generalizes the tableaux formulas of key polynomials 
and Grassmannian stable Grothendieck polynomials.
To prove our description, we construct a new abstract Kashiwara crystal structure 
on set-valued tableaux.
This construction answers an open problem of Monical, Pechenik and Scrimshaw.

\end{abstract}

\maketitle

\section{Introduction}
\label{S: Introduction}
In this paper, we establish a combinatorial rule for Lascoux polynomials
using a combinatorial proof.
Lascoux polynomials, 
denoted by $\fLb$, 
are a $\mathbb{Z}[\beta]-$basis for $\mathbb{Z}[\beta][x_1, x_2, \dots]$ 
indexed by weak compositions
(infinite sequence of non-negative integers with finitely many positive entries).
They are related to
the following polynomials: 
\begin{enumerate}
\item[$\bullet$] Schur polynomials: denoted by $s_\lambda$, which 
are symmetric polynomials in $\mathbb{Z}[x_1, x_2, \dots]$ indexed by partitions
(finite weakly decreasing sequence of positive integers).
They played an important role in representation theory 
of the symmetric group and the general linear group.
\item[$\bullet$] Key polynomials: denoted by $\kappa_\alpha$, which
are polynomials in $\mathbb{Z}[x_1, x_2, \dots]$ indexed by weak compositions.
They were introduced by Demazure in ~\cite{De} for Weyl groups and
are characters of Demazure modules. 
\item[$\bullet$] Grassmannian stable Grothendieck polynomials: 
denoted by $\Gb_\lambda$, which
are polynomials in $\mathbb{Z}[\beta][x_1, x_2, \dots]$ indexed by partitions.
Note that the $\beta$ is also an indeterminant. 
These polynomials are symmetric in the $x$ variables. 
They represent Schubert classes in the connective K-theory of Grassmannians.
\end{enumerate}

The relations between these four polynomials can be described as follows.
\begin{enumerate}
\item[$\bullet$]
Key polynomials generalize Schur polynomials.
More explicitly, assume $\alpha$ is a weak composition
whose first $n$ entries are weakly increasing
and all other entries are 0.
Let $\lambda$ be the partition we get
when we sort $\alpha$ into a weakly decreasing sequence
and remove the trailing 0s. 
Then 
$$
\kappa_\alpha = s_\lambda|_{x_{n+1} = x_{n+2} = \dots = 0}.
$$

\item[$\bullet$]
Grassmannian stable Grothendieck polynomials are K-theoretic analogs of Schur polynomials: 
$G_\lambda^{(0)} = s_\lambda$.
\item[$\bullet$]
Extending this viewpoint,
Lascoux polynomials may be viewed as K-theoretic analogs of key polynomials: 
$\fL_\lambda^{(0)} = \kappa_\lambda$.
\item[$\bullet$]
Lascoux polynomials generalize Grassmannian stable Grothendieck polynomials
in a manner analogous to the generalization
of Schur polynomials by key polynomials.
\end{enumerate}

Their relations are summarized in the following diagram:
\[
\begin{matrix}
\begin{tikzcd}
 \fLb_\alpha \arrow[rr,"\mathrm{specialize}"] \arrow[d,"\beta = 0"]\arrow[d,swap,""] && \Gb_\la \arrow[d,swap,"\beta = 0"] \arrow[d,""]\\
\kappa_\alpha \arrow[rr,swap,"\mathrm{specialize}"] && s_\lambda
\end{tikzcd} \\ \\
\end{matrix}
\]

Here is another perspective to see how Lascoux polynomials
fit into the larger picture.
Lascoux and Sch{\"u}tzenberger found an expansion 
of Schubert polynomials into key polynomials~\cite{LS2}.
This expansion was proved by Reiner and Shimozono ~\cite{RS}.
Grothendieck polynomials are K-theoretic analogs of Schubert polynomials~\cite{LS82}.
Buch, Kresch, Shimozono, Tamvakis and Yong ~\cite{BKSTY} 
proved the stable limit version of this expansion.
They expanded symmerized Grothendieck polynomials 
into Grassmannian stable Grothendieck polynomials.
Finally, Reiner and Yong ~\cite{RY} conjectured an expansion 
of Grothendieck polynomials into Lascoux polynomials, 
generalizing expansions in both ~\cite{RS} and ~\cite{BKSTY}.
Shimozono and Yu ~\cite{SY} proved this conjecture.

Polynomials in the diagram above have tableaux formulas.
Schur polynomials are generating functions
of semistandard Young tableaux (SSYT): 
\begin{equation}
\label{E: SSYT of Schur}
s_\lambda = \sum_T x^{\wt(T)}
\end{equation}
where the sum is over all SSYT with shape $\lambda$
(see \S\ref{S: Background} for relevant definitions).
Lascoux and Sch\"{u}tzenberger 
generalized Equation ~(\ref{E: SSYT of Schur})
by providing a combinatorial formula for key polynomials 
(Equation ~(\ref{E: SSYT rule of Demazure}))
using right keys.
On the other hand, Buch generalized Equation ~(\ref{E: SSYT of Schur})
by establishing a set-valued tableaux (SVT) formula 
for Grassmannian stable Grothendieck polynomials 
(Equation ~(\ref{E: Symmetric SVT rule})).
We generalize all three formulas 
by providing a novel combinatorial formula for Lascoux polynomials
involving both right keys and SVT. 

There already exist various combinatorial formulas of Lascoux polynomials:
\begin{enumerate}
\item [$\bullet$]
Buciumas, Scrimshaw and Weber
~\cite{BSW} established a SVT rule involving the right keys 
and the Lusztig involution, 
which was first conjectured by Pechenik and Scrimshaw ~\cite{PS}.
\item [$\bullet$]
Buciumas, Scrimshaw and Weber 
~\cite{BSW} established a set-valued skyline filling formula, 
which was first conjectured by Monical ~\cite{Mon}.
\item [$\bullet$]
Buciumas, Scrimshaw and Weber 
~\cite{BSW} established
reverse set-valued tableaux rule involving the left keys.
It was then rediscovered by Shimozono and Yu ~\cite{SY}.
This rule can also be rephrased into a form
that involves reverse semistandard Young tableaux.
\item[$\bullet$]
Ross and Yong ~\cite{RossY} conjecture a
rule that involves diagrams.
Their conjectural rule extends the Kohnert diagram
rule for key polynomials. 
In the special case where all positive numbers in
$\alpha$ are the same,
this conjecture is proved in ~\cite{PS}.
\end{enumerate}

We are going to provide another SVT rule 
for Lascoux polynomials (Theorem ~\ref{T: Main}). 
In general, 
our rule and the rule in ~\cite{BSW} sum over different sets of SVT.
Moreover, our rule is easier 
since it does not involve the Lusztig involution. 
In addition, we may view Theorem~\ref{T: Main} from the tableau complex viewpoint ~\cite{KMY}.
For each $\fLb_{\alpha}$, the SVT we summed over form a simplicial complex.
It is a sub-complex of the {\em Young tableau complex} in ~\cite{KMY}.

To prove our result, 
we defined operators $f_i, e_i$ on SVT
and obtain an abstract Kashiwara crystal structure.
Our operators generalize the classical crystal operators 
on SSYT.
However, our construction is not isomorphic 
to the crystal basis of a 
$U_q(\mathfrak{sl_n})-$representation.
In addition, we defined operators $f_i', e_i'$
which can be viewed as ``square roots'' of our $f_i$ and $e_i$.
Notice that Monical, Pechenik and Scrimshaw ~\cite{MPS}
have already defined a crystal structure on SVT,
which comes from a $U_q(\mathfrak{gl_n})-$representation.
However, their crystal operators are not compatible
with $K_+(\cdot)$ introduced in \S\ref{S: right key}.

Our proof mimics Kashiwara's study of Demazure modules 
and crystal basis ~\cite{K}.
Based on our crystal, we define $i$-strings 
similar to ~\cite{K}. 
A key step of our proof is
Corollary ~\ref{C: All, none, or source}, 
which is an analogous result of ~\cite[Proposition ~3.3.5]{K}.
Besides being crucial in the proof, 
our crystal structure is a K-theoretic analogue of the Demazure
crystal introduced in ~\cite{K}.
It can also be viewed as a solution 
to~\cite[Open Problem 7.1]{MPS}
in the context of abstract Kashiwara crystals.

The rest of the paper is organized as follows.
In \S\ref{S: Background}, we will give background.
In \S\ref{S: right key}, we define the right keys
for SVT and 
introduce our main result Theorem ~\ref{T: Main}. 
In \S\ref{S: crystal},
we construct a Kashiwara crystal on SVT
and prove Theorem ~\ref{T: Main}.
In \S\ref{S: K-crystal}, 
we explain why our crystal can be viewed as
a K-analogue of the Demazure
crystal and an answer to ~\cite[Open Problem 7.1]{MPS}.
In \S\ref{S: Atom}, 
we extend our main result
to Lascoux atoms. 

\section{Background}
\label{S: Background}
\subsection{Lascoux Polynomials}
The symmetric group $S_n$ acts on the polynomial ring
$\Z[\beta][x_1,x_2,\dotsc]$ 
by permuting the $x$ variables. 
Let $s_i \in S_n$ denote the transposition
that swaps $i$ and $i+1$.
Following~\cite{LS2} and~\cite{Las01}, 
we define four operators on $\Z[\beta][x_1,x_2,\dotsc]$:
\begin{align*}
\partial_i(f) &= (x_i-x_{i+1})^{-1} (f - s_if) \\
\pi_i(f) &= \partial_i (x_i f) \\
\pab_i(f) &= \partial_i (f + \beta x_{i+1} f) \\
\pib_i(f) &= \pab_i(x_if).
\end{align*}
These four operators satisfy the braid relations.

A weak composition is an infinite sequence of nonnegative integers 
with finitely many positive entries.
When we write a weak composition, 
we ignore the trailing 0s. 
Let $\alpha$ be a weak composition.
We use $\alpha_i$ to denote the $i^{th}$ entry of $\alpha$.
The {\em Lascoux polynomial} $\fLb_\alpha$
is defined by ~\cite{Las}
\begin{align}
\fLb_\alpha = \begin{cases}
x^\alpha & \text{if $\alpha$ is a partition} \\
\pib_i \fLb_{s_i\alpha} &\text{if $\alpha_i<\alpha_{i+1}$.}
\end{cases}
\end{align}
The {\em key polynomial} $\kappa_\alpha$ is defined by
\begin{align}
\kappa_\alpha = \fLb_{\alpha}|_{\beta=0}.
\end{align}

\subsection{Tableaux}
In this subsection, 
we define a {\em tableau} as a filling of a diagram $\lambda/\mu$
with $\mathbb{Z}_{>0}$.
A tableau has \emph{normal} (resp. \emph{antinormal}) shape 
if it is empty or has a unique northwestmost (resp. southeastmost) corner.
A {\em semistandard Young tableau} (SSYT)
is a tableau whose columns are strictly increasing and
rows are weakly increasing. 
Let $T$ be a SSYT.
The {\em weight} of $T$, denoted by $\wt(T)$, is a weak composition
whose $i^{th}$ entry is the number of $i$
in $T$.
The {\em column order} is a total order on cells of $T$.
It goes from left to right and
from bottom to top within each column. 
The column word of $T$, denoted by $\word(T)$,
is the word we get if we read the number in each cell of $T$
in the column order.

A {\em key} is a SSYT with normal shape 
such that each number in the $j^{th}$
column also appears in the $(j-1)^{th}$ column.
There are natural bijections
between weak compositions and keys. 
Let $\key(\cdot)$ be the map that sends the weak composition
to its corresponding key. 
Its inverse map is simply $\wt(\cdot)$.
For instance,
$$
\key(1,0,3,2) = 
\raisebox{0.3cm}{
\begin{ytableau}
1 & 3 & 3 \cr
3 & 4\cr
4 \cr 
\end{ytableau}}
$$

The Knuth equivalence $\sim$ is defined on the set of all words by
the transitive closure of 
$$
uxzyv \:\sim\: uzxyv \textrm{ if } x \leq y < z, 
$$
$$
uyxzv \:\sim\: uyzxv \textrm{ if } x < y \leq z,
$$
where $u$ and $v$ are words. 
From ~\cite{F}, for each SSYT $T$, 
there exists a unique SSYT $T^{\searrow}$
with antinormal shape
such that $\word(T) \sim \word(T^\searrow)$.

Each SSYT $T$ with normal shape is associated with a key 
called the {\em right key}.
It has the same shape as $T$ and is denoted by $K_+(T)$.
Let $T_{\geq j}$ be the tableau we get if we remove the first $j-1$ columns of $T$.
Then column $j$ of $K_+(T)$ is defined as 
the rightmost column of $T_{\geq j}^\searrow$.
In \S\ref{S: right key}, 
we will describe an easier way to compute $K_+(T)$.  

\begin{exa}
\label{X: right key def}
Let $T$ be the following SSYT:
\begin{align*}
\begin{ytableau}
1 & 2 & 4 & 7 \cr
3 & 5 & 6\cr
4 & 8\cr 
6  
\end{ytableau}
\end{align*}
Then $T_{\geq 1} = T$. 
Consider the following SSYT $T'$
with antinormal shape:
\begin{align*}
\begin{ytableau}
\none & \none & \none & 2 \cr
\none & \none & 3 & 4 \cr
\none & 1 & 5 & 7 \cr
4 & 6 & 6 & 8 
\end{ytableau}
\end{align*}
Notice that $\word(T) = 6431852647 \sim 4616538742 = \word(T')$,
so $T' = T^\searrow$.
Thus, column 1 of $K_+(T)$ consists of $\{ 2,4,7,8\}$.
Similarly, $T_{\geq 2}^\searrow, T_{\geq 3}^\searrow$ and $T_{\geq 4}^\searrow$ are 
\begin{align*}
\begin{ytableau}
\none & \none & 4 \cr
\none & 2 & 7 \cr
5 & 6 & 8 \cr
\end{ytableau}
\quad\quad\quad
\begin{ytableau}
\none & 4 \cr
6 & 7 \cr
\end{ytableau}
\quad\quad\quad
\begin{ytableau}
7\cr
\end{ytableau}
\end{align*}

Thus, $K_+(T)$ is 
\begin{align*}
\begin{ytableau}
2 & 4 & 4 & 7 \cr
4 & 7 & 7\cr
7 & 8\cr 
8 \cr 
\end{ytableau}
\end{align*}
\end{exa}

Finally, we can introduce a well-known combinatorial
rule of key polynomials ~\cite{LS1, LS2}.
Let $\alpha$ be a weak composition. 
Let $\SSYT(\alpha)$ be the set of all SSYT such that $T$ has the same shape
as $\key(\alpha)$ and
$K_+(T) \leq \key(\alpha)$ where the comparison is entry-wise.
Then
\begin{equation}
\label{E: SSYT rule of Demazure}
\kappa_{\alpha}= 
\sum_{T \in \SSYT(\alpha)} x^{\wt(T)}.
\end{equation}

\subsection{Abstract Kashiwara crystal}

First, we will introduce {\em Abstract Kashiwara} {\em crystals} ~\cite{Kas90}~\cite{Kas91} 
following~\cite{BS}.
\begin{defn} ~\cite[Definition 2.13]{BS}
An {\em abstract Kashiwara} $\textrm{GL}_n${\em-crystal} 
is a nonempty set $\B$
together with the following maps:
$$
e_i, f_i : \B \rightarrow \B \sqcup \{ \emt \},
$$
$$
\varepsilon_i, \varphi_i: \B \rightarrow \mathbb{Z}
\sqcup \{ - \infty\},
$$
$$
\wt : \B \rightarrow \mathbb{Z}^n,
$$
where $i \in [n-1]$, 
satisfying the following two conditions. 
\begin{enumerate}
\item[K1:] For all $X, Y \in \B$,
we have 
$e_i(X) = Y$ if and only if
$f_i(Y) = X$.
If this is the case then
\begin{equation*}
\begin{split}
\varepsilon_i(Y) & = \varepsilon_i(X) - 1, \\
\varphi_i(Y) & = \varphi_i(X) + 1, \\
\wt(Y) & = \wt(X) + v_i - v_{i+1},
\end{split}
\end{equation*}
where $v_1, \dots, v_n$ is the standard basis of $\mathbb{Z}^n$. 

\item[K2:] For all $X \in \B$,
we have
$$
\varphi_i(X) = \langle \wt(X), v_i - v_{i+1} \rangle + \varepsilon_i(X).
$$
\end{enumerate}

Furthermore, $\B$ is called {\em seminormal} if
$$
\varepsilon_i(X) = \max\{k: e_i^k(X) \neq \emt \}
\quad \textrm{ and }\quad 
\varphi_i(X) = \max\{k: f_i^k(X) \neq \emt \}
$$
for all $X \in \B$ and $i \in B_{n-1}$.
\end{defn}

\begin{defn}~\cite{K}
Let $\B$ be an abstract Kashiwara $\textrm{GL}_n$-crystal.
For each $i \in [n-1]$, 
an {\em $i$-string} is a sequence
$X_0, \dots, X_k \in \B$
satisfying:
\begin{enumerate}
\item[$\bullet$] $e_i(X_0) = f_i(X_k) = \emt$
\item[$\bullet$] $f_i(X_j) = X_{j+1}$
for each $j \in \{0, 1, \dots, k-1 \}$.
\end{enumerate}
We say $X_0$ is the {\em source} of its string.
Diagrammatically, we can represent the string as:

$$
X_0 \xrightarrow{i} 
X_1 \xrightarrow{i} X_2\xrightarrow{i}\cdots 
\xrightarrow{i} X_k
$$

\end{defn}

It is clear that $\B$ can be broken into
a disjoint union of $i$-strings for each $i$. 
If we know $\B$ is seminormal, 
then we have the following well-known
result regarding the weight of
elements in an $i$-string.
\begin{lemma}
\label{L: string weight}
Let $\B$ be a seminormal abstract 
Kashiwara $\textrm{GL}_n$-crystal.
Consider the $i$-string 
$X_0, \dots, $ $X_k$
for some $i \in [n-1]$. 
Then $\wt(X_j) = s_i \wt(X_{k-j})$
for each $j \in \{0, 1, \dots, k \}$,
where $s_i$ is the operator that
swaps the $i^{th}$ entry 
and the $(i+1)^{th}$ entry. 
\end{lemma}
\begin{proof}
Since $\B$ is seminormal,
$\varepsilon_i(X_0) = 0$
and $\varphi_i(X_0) = k$.
Let $a$ be the $(j+1)^{th}$
entry of $\wt(X_0)$.
By (K2), $\wt(X_0) = (\dots, a + k, a, \dots)$.
Apply $f_i$ for $j$ times on $X_0$
and obtain $X_j$.
By (K1), $\wt(X_j) = 
(\dots,a + k - j, a + j, \dots)$.
Thus, $\wt(X_j) = s_i \wt(X_{k-j})$.
\end{proof}

Now we describe a well-known example of 
an abstract Kashiwara crystal.
Let $\B(\lambda,n)$ be the set of all SSYT whose shapes
are $\lambda$ and entries are in $[n]$.
Take $T \in \B(\lambda, n)$ and consider its column word.
We replace each $i$ by $``)''$
and replace each $i + 1$ by $``(''$.
Then we remove all other numbers. 
The resulting word is called the {\em $i$-word } of $T$.
We may pair $``(''$ with $``)''$ in the usual way.
\begin{defn}
Define $\epsilon_i(T)$ as the number of unpaired $``(''$
and $\phi_i(T)$ as the number of unpaired $``)''$.

If $\phi_i(T) = 0$, then $f_i(T) := \emt$.
Otherwise, we can find the $i$ in $T$ that corresponds
to the last unpaired $``)''$ in the $i$-word.
We change this $i$ into $i+1$ and get $f_i(T)$.

If $\epsilon_i(T) = 0$, then $e_i(T) := \emt$.
Otherwise, we can find the $i+1$ in $T$ that corresponds
to the first unpaired $``(''$ in the $i$-word.
We change this $i+1$ into $i$ and get $e_i(T)$.
\end{defn}

It is a well-known result that $B(n, \lambda)$, together
with $e_i, f_i, \phi_i, \epsilon_i$ and $wt$, 
form a seminormal abstract Kashiwara $\textrm{GL}_n$-crystal.
Moreover, they correspond to the crystal basis from the 
irreducible highest weight $U_q(\mathfrak{gl_n})$
module of highest weight $\lambda$.

We can use the operator $f_i$ to compute $\SSYT(\alpha)$.
Let $S$ be a subset of $B(n, \lambda)$.
Define $\F_i S$ as $\{ (f_i)^j (T): T \in S, j \geq 0\} - \{ \emt\}$.
\begin{thm}[\cite{K}]
\label{T: obtain SSYT by f_i}
Let $\alpha$ be a weak composition
such that $\alpha^+ = \lambda$
and $\alpha_i > 0$ for $i > n$.
We can write $\alpha$ as 
$s_{i_1}\dots s_{i_k} \lambda$,
where $k$ is minimized.
Then we have
\begin{equation*}
\SSYT(\alpha)
= \F_{i_1} \dots \F_{i_k} \{ u_\lambda\}.
\end{equation*}
Here, $u_\lambda$ is the SSYT with shape $\lambda$
such that its $r^{th}$ row only has $r$.
\end{thm}

$\SSYT(\alpha)$, together with the maps,
is known as a {\em Demazure crystal}.

\subsection{Set-valued Tableaux}
We start to view a {\em tableau} as a filling
where entries are finite non-empty subsets of $\Z_{>0}$.
\begin{defn}
A {\em set-valued tableau} (SVT) is a tableau
such that no matter how we pick one entry in each set,
the resulting tableau is a SSYT. 
Let $T$ be a SVT. 
Define $\cS(T)$ to be the set of SSYT
obtained by picking one number in each cell of~$T$.
\end{defn}
\begin{exa}
\label{E: SVT}
The following $T$ is a SVT:

\begin{align*}
T = \raisebox{0.3cm}{
\begin{ytableau}
1 & 13 & 36 \cr
23 & 47\cr
567 \cr 
\end{ytableau}} \: ,
\end{align*}
where $23$ represents the set $\{2,3\}$.
The set $\cS(T)$ consists of 48 SSYT, 
including \begin{align*}
\raisebox{0.3cm}{
\begin{ytableau}
1 & 3 & 6 \cr
3 & 7\cr
7 \cr 
\end{ytableau}} 
\quad \textrm{ and } \quad
\raisebox{0.3cm}{
\begin{ytableau}
1 & 3 & 6 \cr
2 & 4\cr
5 \cr 
\end{ytableau}} \: .
\end{align*}

The following example is not a SVT
\begin{align*}
\begin{ytableau}
1 & 14 & 46 \cr
23 & 47\cr
567 \cr 
\end{ytableau}
\end{align*}
since if we pick 4 in
both cells of column 2, 
the resulting filling cannot be a SSYT.
\end{exa}

\begin{rem}
A SSYT can be viewed as a
SVT where each set is a singleton.
\end{rem}

\begin{defn}
Let $T$ be a SVT of shape $\lambda$.
Let $\wt(T)$ be the weak composition whose
$i^{th}$ entry is the number of 
$i$'s in $T$.
Let $\ex(T)$ be the number 
$|\wt(T)| - |\lambda|$.
\end{defn}

It is clear that the definition of $\wt(\cdot)$
agrees with our previous definition
when every set in $T$ is a singleton.
Intuitively, $\ex(T)$ counts the number of ``extra" 
numbers in $T$.

To supplement our introduction, 
before continuing our development, 
we state an application of SVT.
This is a 
SVT rule for
Grassmannian stable Grothendieck polynomials $\Gb_\lambda$ due to Buch~\cite{B:Gr}.
Instead of defining $\Gb_\lambda$, 
we restate its relation with Lascoux polynomials.
Assume $\alpha$ is a weak composition whose
first $n$ entries are weakly increasing
and all other entries are all 0.
Sort $\alpha$ and obtain the partition $\lambda$.
Then 
$$
\fLb_\alpha = \Gb_\lambda|_{x_{n+1} = x_{n+2} = \dots = 0}.
$$
\begin{thm}[\cite{B:Gr}]
Let $\lambda$ be a partition. Then
\begin{equation}
\label{E: Symmetric SVT rule}
\Gb_\lambda|_{x_{n+1} = x_{n+2} = \dots = 0}= \sum_T \beta^{\ex(T)}x^{\wt(T)} 
\end{equation}
where the sum is over all SVT $T$
whose shape is $\lambda$ and whose
entries are subsets of ~$[n]$.
\end{thm}

\section{The right keys}
\label{S: right key}

In this section, 
we first describe a direct way to compute $K_+(T)$
for SSYT $T$ with normal shape.
Then we generalize the right key to all SVT with normal shape.
Finally, we introduce our main result. 

\subsection{Compute right keys using the star operator}
Shimozono and Yu ~\cite{SY} used the following operator
to compute right keys.\footnote{Notice that we replace ``smallest'' by ``largest''
and ``at least'' by ``at most''.
This is because ~\cite{SY} focused on reverse SSYT
(tableaux whose rows are weakly decreasing and columns
are strictly decreasing)
while this paper focused on SSYT.} 
This method is a reformulation of Willis' method ~\cite{W}.
\begin{defn} 
First, we define $S \star m$ for
$S \subseteq \mathbb{Z}$
and $m\in \Z$. 
Let $m'$ be the largest number in $S$ such that 
$m' \leq m$.
If $m'$ does not exist, 
we let $S \star m = S \sqcup \{m \}$.
Otherwise,
we define $S \star m = (S - \{ m'\})
\sqcup \{ m \}$.

More generally, we may define
$\star$ to be a right action 
of the free monoid of words with characters in the 
set $\Z$, 
on the power set of $\Z$.
If $w = w_1 \cdots w_n$ is a word of integers, 
we define $S \star w =
(\cdots((S \star w_1) \star w_2) \cdots \star w_n)$,
and $S \star w = S$ 
if $w$ is the empty word.
\end{defn}
\begin{exa}
We have
$$
\{2,4,5,7\} \star 3462 = \{2,3,4,6,7\},
$$
$$
\{2,4,5,7\} \star 1284 = \{1,2,4,5,8\}.
$$
\end{exa}

\begin{rem} 
\label{R: Knuth and star}
Similar to ~\cite[Remark 4.7]{SY},
we have $S \star w = S \star w'$,
if $w$ and $w'$ are Knuth equivalent.
\end{rem}

We have the following way to compute a right key. 
\begin{lemma}
\label{L: star computes key}
Column $j$ of $K_+(T)$ consists of
$\emptyset \star \word(T_{\geq j})$. 
\end{lemma}
\begin{proof}
By definition, column $j$ of $K_+(T_{\geq j})$
equals the last column of $T_{\geq j}^\searrow$.
Since $T_{\geq j}^\searrow$ has antinormal shape, 
$\emptyset \star \word(T_{\geq j}^\searrow)$
is the set of numbers 
in the last column of $T_{\geq j}^\searrow$.
Then the proof is finished by $\word(T_{\geq j}) \sim \word(T_{\geq j}^\searrow)$ 
and Remark ~\ref{R: Knuth and star}.
\end{proof}

\begin{exa}
Let $T$ be the following SSYT:
\begin{align*}
\begin{ytableau}
1 & 2 & 4 & 7 \cr
3 & 5 & 6\cr
4 & 8\cr 
6 \cr 
\end{ytableau}
\end{align*}
Then column 1 of $K_+(T)$ consists of
$\emptyset \star 6431852647 = \{2,4,7,8\}$.
Column 2, 3 and 4 of $K_+(T)$ consist of:
$\emptyset \star 852647 = \{4,7,8\}$,
$\emptyset \star 647 = \{4,7\}$ and
$\emptyset \star 7 = \{7\}$.
Thus, $K_+(T)$ is
\begin{align*}
\begin{ytableau}
2 & 4 & 4 & 7 \cr
4 & 7 & 7\cr
7 & 8\cr 
8 \cr 
\end{ytableau}
\end{align*}
which agrees with Example ~\ref{X: right key def}.
\end{exa}

\subsection{Generalizing \boldmath{$K_+(\cdot)$} to SVT}

In this subsection, 
we assign a SSYT to each SVT with normal shape.
Then we explains why this assignment 
naturally generalizes $K_+(\cdot)$.

\begin{defn}
Let $T$ be a SVT with normal shape.
Define 
$$T_{\max} \df 
\max_{P \in \cS(T)}(K_+(P))$$
where $\max$ is entry-wise.
\end{defn}

\begin{exa}
We start with the SVT $T$.
The set $\cS(T)$ has two SSYT.

$$
T = 
\raisebox{0.1cm}{
\begin{ytableau}
1 & 23\\
3 
\end{ytableau}}
\:,\quad
\cS(T) = \{
\raisebox{0.1cm}{
\begin{ytableau}
1 & 2\\
3 
\end{ytableau}}
\:, \quad
\raisebox{0.1cm}{
\begin{ytableau}
1 & 3\\
3 
\end{ytableau}} \: 
\} \: .
$$
We compute the right keys
of the two tableaux in $\cS(T)$ 
and get:
$$
\begin{ytableau}
2 & 2\\
3 
\end{ytableau}
\quad \textrm{ and } \quad
\begin{ytableau}
1 & 3\\
3 
\end{ytableau} \: .
$$
Take the maximum of each entry
and obtain:
$$
T_{\max} = 
\begin{ytableau}
2 & 3\\
3 
\end{ytableau} \: .
$$
\end{exa}

\begin{rem}
Readers might wonder whether 
$T_{\max}$ can be computed as follows:
Pick the largest number in each entry
and compute the right key of this 
SSYT.
The previous example shows that 
this approach does not work. 
If we pick the largest number in each
entry,
we obtain 
$$
\begin{ytableau}
1 & 3\\
3 
\end{ytableau}
$$
whose right key is not $T_{\max}$.
\end{rem}

From the definition of $T_{\max}$,
it is an entry-wise maximum of several SSYT.
Thus, $T_{\max}$ is also a SSYT.
Next, we find an easier way
to compute $T_{\max}$ and show it is a key.
We start with a definition.

\begin{defn}
For finite $S \subseteq \mathbb{Z}_{>0}$,
let $\word(S)$ be the word we get
if we list numbers of $S$ in increasing order.
For a SVT $T$, 
let $\word(T) := \word(S_1) \cdots \word(S_n)$,
where $S_1, \dots, S_n$ are entries of $T$ 
in the column order.
\end{defn}

Now we may introduce an easier way to compute $T_{\max}$:
\begin{lemma}
\label{L: T_max}
Let $T$ be a normal SVT. 
Column $j$ of $T_{\max}$ consists of
$\emptyset \star \word(T_{\geq j})$,
where $T_{\geq j}$ is obtained by
removing the first $j-1$ columns of $T$.
\end{lemma}

\begin{exa}
Let $T$ be the SVT in example ~\ref{E: SVT}.
Then $\word(T) = 567231471336$.
Column 1 of $T_{\max}$ consists of 
$\emptyset \star 567231471336 = \{3,6,7\}$.
Column 2 and 3 of $T_{\max}$
consist of $\emptyset \star 471336 = \{6,7\}$
and $\emptyset \star 36 = \{6\}$.
Thus, 
\begin{align*}
T_{\max} = 
\raisebox{0.3cm}{
\begin{ytableau}
3 & 6 & 6 \cr
6 & 7\cr
7 \cr 
\end{ytableau}} 
\: .    
\end{align*}

We may check this agrees
with the definition of $T_{\max}$.
First, we compute the right keys of
the two SSYT from $\cS(T)$ 
in example ~\ref{E: SVT}.
We get 
\begin{align*}
\raisebox{0.3cm}{
\begin{ytableau}
1 & 6 & 6 \cr
6 & 7\cr
7 \cr 
\end{ytableau}} 
\quad \textrm{ and } \quad
\raisebox{0.3cm}{
\begin{ytableau}
3 & 3 & 6 \cr
4 & 6\cr
6 \cr 
\end{ytableau}} \: ,
\end{align*}
whose entry-wise maximum
is the key above. 
The right keys of the other 46 SSYT
in $\cS(T)$
are entry-wise less than or equal to this key.
\end{exa}

To prove the lemma, 
we need the entry-wise maximum of sets:
\begin{defn}
Let $C$ be a finite collection of sets
such that all sets in $C$ have the same size $k$.
We may view each element of $C$ as a 
column of a SSYT and take the entry-wise maximum. 
Then $\max_{S \in C}S$ is the set corresponding
to the resulting column. 
More explicitly, 
$\max_{S \in C}S$ is the set with size $k$ 
such that its $i^{th}$ smallest number
is $$\max_{S \in C}(i^{th} \textrm{ smallest number in } S).$$
\end{defn}

\begin{proof}[Proof of Lemma~\ref{L: T_max}]
It is enough to assume $j = 1$.
By the definition of $T_{\max}$ 
and Lemma~\ref{L: star computes key},
column 1 of $T_{\max}$ consists of
$\max_{P \in \cS(T)} \emptyset \star \word(P)$.
Thus, we need to prove
\begin{align}
\label{E: SVT star}
\max_{P \in \cS(T)} \emptyset \star \word(P)
= \emptyset \star \word(T)    
\end{align}

First, we prove~(\ref{E: SVT star}) for $T$ 
that only has one column.
Let $S_1, \dots, S_k$ be the entries of $T$,
enumerated from bottom to top. 
We have $\min(S_i) > \max(S_{i+1})$
for $1 \leq i \leq k-1$.
As $P$ ranges over $\cS(T)$,
$\word(P)$ ranges over $s_1 \cdots s_k$ 
with $s_i \in S_i$.
Thus, $\emptyset \star \word(P)$
ranges over 
$\{ s_1 > \cdots > s_k \}$
with $s_i \in S_i$.
The left hand side of~(\ref{E: SVT star})
is $\{ \max(S_1)> \cdots > \max(S_k)\}$.
For the right hand side, 
notice that $\emptyset \star \word(S_1) 
= \{ \max(S_1)\}$.
Since $\max(S_1) > \max(S_2)$,
$\emptyset \star \word(S_1) \word(S_2)
= \{ \max(S_1), \max(S_2) \}$.
A simple induction on $i$ would yield
$$
\emptyset \star \word(S_1) \cdots \word(S_i)
= \{\max(S_1)>  \cdots > \max(S_i)\}.
$$
Since $\word(T) = \word(S_1) \cdots \word(S_k)$,
the right hand side of~(\ref{E: SVT star})
is $\{ \max(S_1)$ $> \cdots> \max(S_k)\}$.
We have established~(\ref{E: SVT star})
for $T$ with one column.

Now we prove~(\ref{E: SVT star})
for all SVT $T$.
We perform an induction on the number of
entries of $T$ that are not in column 1. 
For the base case, 
we assume $T$ has no such entries. 
In other words, $T$ has only one column.
This case is checked above. 

Now assume $T$ has more than one column.
Let $X$ be the highest entry 
in the rightmost column of $T$.
We may remove $X$ from $T$ and 
raise all entries below $X$.
The resulting filling, $T'$,
is clearly a SVT. 
We have
\begin{align*}
\word(T) & = \word(T') \word(X) 
\quad \quad \textrm{ and }\\
\{\word(P): P \in \cS(T)\}
& = \{ \word(P')x: P'\in \cS(T'), x \in X \}.    
\end{align*}
Our goal~(\ref{E: SVT star}) becomes
\begin{align}
\label{E: SVT star 2}
\max_{P' \in \cS(T'), x \in X} 
\emptyset \star \word(P')x
= \emptyset \star \word(T')\word(X).  
\end{align}

To show this equality, 
we first find an alternative way to 
write its right hand side. 
By the inductive hypothesis,
$$
\max_{P' \in \cS(T')} \emptyset \star \word(P')
= \emptyset \star \word(T').  
$$

Use $\{a_1 < \cdots < a_k\}$ to denote 
$\emptyset \star \word(T')$.
We know $k = |\emptyset \star \word(P')|$
for any $P' \in \cS(T')$.
Thus, $k$ is the number of entries in
column 1 of $K_+(P')$,
which is also the number of rows in $T'$ and $T$.
Consequently, 
$k = |\emptyset \star \word(P)|$
for any $P \in \cS(T)$.

Next, we show $\min(X) \geq a_1$ by contradiction.
Assume there exists $x \in X$ with $x < a_1$.
We may pick $P' \in \cS(T')$
such that $\min(\emptyset \star \word(P')) = a_1$.
Then consider the tableau $P \in \cS(T)$
with $\word(P) = \word(P')x$.
We have 
$\emptyset \star \word(P) 
= (\emptyset \star \word(P')) \star x$,
which has more than $k$ numbers.
Contradiction. 

Since $\min(X) \geq a_1$,
we may partition $X$ 
as $X_1 \sqcup \cdots \sqcup X_k$
by $X_i = X \cap [a_i, a_{i+1})$,
where $a_{k+1} = \infty$ by convention.
Consider the action of $\word(X) = 
\word(X_1) \cdots \word(X_k)$
on $\{a_1, \dots, a_k\}$.
When $X_i$ acts, 
$a_i$ is still in the set.
If $X_i$ is non-empty,
$a_i$ will be bumped by $\min(X_i)$,
which is then bumped by the second
smallest number in $X_i$.
Eventually, the action of $X_i$
replaces $a_i$ by $\max(X_i)$.
Thus, $\{a_1, \dots, a_k\} \star \word(X) 
= \{\overline{a}_1< \cdots < \overline{a}_k\}$,
where $\overline{a}_i = \max(X_i)$
if $X_i \neq \emptyset$ 
and $\overline{a}_i = a_i$ otherwise. 

We have turned the right hand side 
of~(\ref{E: SVT star 2})
into $\{\overline{a}_1< \cdots < \overline{a}_k\}$.
It remains to establish the following two statements:
\begin{itemize}
\item For any $P' \in \cS(T'), x \in X$ and $1 \leq i \leq k$,
the $i^{th}$ smallest number 
of $\emptyset \star \word(P')x$ is 
at most $\overline{a}_i$.
\item For any $1 \leq i \leq k$,
we may find $P' \in \cS(T')$ and $x \in X$
such that the $i^{th}$ smallest number 
of $\emptyset \star \word(P')x$
achieves $\overline{a}_i$.
\end{itemize}

Now we prove these two claims. 
\begin{itemize}
\item 
Take any $P' \in \cS(T')$ and $x \in X$.
Let $\{b_1 < \dots < b_k\} = 
\emptyset \star \word(P')$.
Our inductive hypothesis
implies $b_i \leq a_i$
for all $1 \leq i \leq k$.
Now, assume $x$ bumps $b_j$
when acting on $\{b_1 < \dots < b_k\}$,
becoming the $j^{th}$ smallest number
in the resulting set.
We only need to check $x \leq \overline{a}_j$.
Notice that $x < b_{j+1} \leq a_{j+1}$
with $b_{k+1} = \infty$ by convention.
Thus, $x \in X_1 \sqcup \cdots \sqcup X_j$.
If $X_j \neq \emptyset$,
$\overline{a}_j = \max(X_j) \geq x$.
Otherwise, $x \in X_1 \sqcup \cdots 
\sqcup X_{j-1}$, 
so $x < a_j = \overline{a}_j$.
\item
Take $1 \leq i \leq k$.
First, assume $X_i \neq \emptyset$.
By the inductive hypothesis,
we may pick $P' \in \cS(T')$ 
such that if we let 
$\{b_1 < \dots < b_k\} = 
\emptyset \star \word(P')$,
then $b_{i+1} = a_{i+1}$.
Pick $x = \max(X_i)$,
so $b_{i+1} = a_{i+1} > 
x \geq a_i \geq b_i$.
When $x$ acts on $\{b_1 < \dots < b_k\}$,
it will bump the $b_i$.
The $i^{th}$ smallest number
in the resulting set is 
$x = \overline{a}_i$.
Finally, assume $X_i = \emptyset$,
so $\overline{a}_i = a_i$.
Pick $P' \in \cS(T')$ 
such that if we let 
$\{b_1 < \dots < b_k\} = 
\emptyset \star \word(P')$,
then $b_{i} = a_{i}$.
Pick any $x \in X$.
If $x < a_i$,
$x$ will not bump $b_i$
when acting on $\{b_1 < \dots < b_k\}$.
Otherwise, we know $x \geq a_{i+1}$
because $X_i = \emptyset$.
Since $b_{i+1} \leq a_{i+1} \leq x$,
$x$ will not bump $b_i$
when acting on $\{b_1 < \dots < b_k\}$.
In either case, 
the $i^{th}$ largest number of 
$\{b_1 < \dots < b_k\} \star x$
remains to be $b_i = \overline{a}_i$.

\end{itemize}

\end{proof}

\begin{cor}
$T_{\max}$ is a key .
\end{cor}
\begin{proof}
Let $j$ be a positive integer.
By Lemma ~\ref{L: T_max}, 
it remains to show 
$$
\emptyset \star \word(T_{\geq j})
\supseteq \emptyset \star \word(T_{\geq j+1}).
$$
This is implied by ~\cite[Lemma 4.9]{SY} .
\end{proof}

\begin{defn}
The right key of a SVT $T$
is $K_+(T) := T_{\max}$.
Let $\SVT(\alpha)$ be the set of all $T$
such that $K_+(T) \leq \key(\alpha)$.
\end{defn}

\begin{rem}
By the definition of $K_+(\cdot)$,
we may also describe $\SVT(\alpha)$ as 
all SVT $T$ such that 
$\cS(T) \subseteq \SSYT(\alpha)$.
\end{rem}

We end this section by introducing our main result:
\begin{thm}
\label{T: Main}
Let $\alpha$ be a weak composition.
Then
\begin{equation}
\label{E: Main}
\fLb_\alpha = \sum_{T \in \SVT(\alpha)} x^{\wt(T)}\beta^{\ex(T)}.
\end{equation}
\end{thm}
\begin{exa}
Let $\alpha = (1,0,2)$.
Then $\SVT(\alpha)$ consists of the following:

\begin{ytableau}
1 & 1\\
2 
\end{ytableau}
\quad
\begin{ytableau}
1 & 2\\
2 
\end{ytableau}
\quad
\begin{ytableau}
1 & 1\\
3 
\end{ytableau}
\quad
\begin{ytableau}
1 & 3\\
2 
\end{ytableau}
\quad
\begin{ytableau}
1 & 3\\
3 
\end{ytableau}
\\[1mm]

\begin{ytableau}
1 & 12\\
2 
\end{ytableau}
\quad
\begin{ytableau}
1 & 1\\
23 
\end{ytableau}
\quad
\begin{ytableau}
1 & 13\\
2 
\end{ytableau}
\quad
\begin{ytableau}
1 & 3\\
23 
\end{ytableau}
\quad
\begin{ytableau}
1 & 13\\
3 
\end{ytableau}
\quad
\begin{ytableau}
1 & 23\\
2 
\end{ytableau}
\\[1mm]

\begin{ytableau}
1 & {\small 123}\\
2 
\end{ytableau}
\quad
\begin{ytableau}
1 & 13\\
23 
\end{ytableau}

Thus, we may write $\fLb_{(1,0,2)}$ as
\begin{equation*}
\begin{split}
&x_1^2x_2 + x_1x_2^2 + x_1^2x_3 
+ x_1x_2x_3 + x_1x_3^2 \\
+ & \beta(x_1^2x_2^2 + 2x_1^2x_2x_3 + x_1x_2x_3^2
+ x_1^2x_3^2 + x_1x_2^2x_3)\\
+ & \beta^2(x_1^2x_2^2x_3 + x_1^2x_2x_3^2)
\end{split}
\end{equation*}
\end{exa}

Equation ~(\ref{E: Main}) generalizes 
the two combinatorial rules in \S\ref{S: Background}:
\begin{enumerate}
\item[$\bullet$] If we set $\beta = 0$, then the left hand side of
Equation ~(\ref{E: Main}) becomes $\kappa_\alpha$.
In the right hand side, 
only $T$ with $\ex(T) = 0$ can survive in the sum.
Clearly, $\{T \in \SVT(\alpha): \ex(T) = 0\} = \SSYT(\alpha)$.
Thus, our rule extends Equation ~(\ref{E: SSYT rule of Demazure}).
\item[$\bullet$] Assume $\alpha$ is a weak composition whose
first $n$ entries are weakly increasing
and the other entries are all 0.
In each column of $\key(\alpha)$, 
there are $n, n-1, n-2, \dots$.
Then $T \in \SVT(\alpha)$ if and only if $T$ has the shape $\alpha^+$
and entries of $T$ are subsets of $[n]$.
Thus, our rule extends Equation ~(\ref{E: Symmetric SVT rule}).
\end{enumerate}

We will prove Theorem ~\ref{T: Main}
in the next section. 

\section{Abstract Kashiwara crystals 
on SVT} 

\label{S: crystal}

As stated in \S\ref{S: Background},
$\SSYT(\alpha)$ can be computed using
the $f_i$ operators. 
To prove our result, we can construct 
an abstract Kashiwara $\textrm{GL}_n$-crystal
on the set of SVT with entries in $[n]$.

\subsection{Constructing an abstract 
Kashiwara crystal on SVT}

Let $\B_n$ be the set of SVT
with entries in $[n]$. 
We would like to turn $\B_n$
into an abstract Kashiwara crystal.
First, we let $\wt(\cdot)$ be the weight function 
on SVT defined above. 
Next, we would like to generalize the $i$-word
defined on SSYT in \S\ref{S: Background}.

\begin{defn}
Let $\B_\word$ be the set of finite words
generated by $``("$, $``)"$, and $``)-("$
under concatenation.

Take $T \in \B_n$ and $i \in [n-1]$. 
The {\em $i$-word} of $T$
is the following element of $\B_\word$:

Read through entries of $T$ in the column order. 
Whenever we see a set containing
$i$ but not $i+1$,
we write $``)"$.
Whenever we see a set containing
$i+1$ but not $i$,
we write $``("$.
Whenever we see a set containing
$i$ and $i + 1$,
we write $``)-("$. 
\end{defn}

\begin{exa}
\label{E: i-word}
Consider the following element from $\B_4$
$$
T = 
\raisebox{0.3cm}{
\begin{ytableau}
1 & 1 & 2 & 23\cr
2 & 23\cr
34 \cr 
\end{ytableau}}\: .    
$$
It has 1-word $()()(($.
It has 2-word $())-())-($.
It has 3-word $)-((($.
\end{exa}

To define an abstract Kashiwara 
$\textrm{GL}_n$-crystal on $\B_n$,
we first need some definitions on $\B_\word$.
Take $w \in \B_\word$.
Ignore its $``-"$ and pair 
the $``("$ with $``)"$ in the usual way. 
Then we construct an equivalence
relation on all characters.
This relation is generated by
the following two requirements. 
\begin{enumerate}
\item[$\bullet$] If an $``("$ is paired with $``)"$,
then these two characters
and everything between them
should be in the same class.
\item[$\bullet$] For each $``)-("$, 
these three characters are in the 
same class.
\end{enumerate}
It is easy to see that each 
equivalence class is a contiguous 
subword. 
For instance, 
the first word 
$))-(())-())-(()-($ 
is partitioned into 4 classes:
$$
\textcolor{myblue}{)} \quad\quad\textcolor{myblue}{)-(())-()} \quad\quad\textcolor{mypurple}{)-(}
\quad\quad\textcolor{myred}{()-(}
$$

Notice that any unpaired $``)"$ 
must be the first character in its class.
Any unpaired $``("$ must be the 
last character in its class. 
Thus, we may classify each class
by whether it starts with an unpaired $``)''$
and whether it ends with an unpaired $``(''$.
\begin{enumerate}
\item[$\bullet$] {\em null form}: This class
does not have unpaired $``(''$ or $``)''$.
For example, $``(()-())-()''$.
\item[$\bullet$] {\em left form}: This class
does not have unpaired $``)''$ but
ends with an unpaired $``(''$.
For example, $``(())-(''$.
\item[$\bullet$] {\em right form}: This class
does not have unpaired $``(''$ but
starts with an unpaired $``)''$.
For example, $``)-()-()''$.
\item[$\bullet$] {\em combined form}: This class
start with an unpaired $``)''$ and
ends with an unpaired $``(''$.
For example, $``)-()-(())-(''$.
\end{enumerate}
In the previous example, 
the first two classes are 
right forms.
The third class is a combined form
and the last class is a left form.
In general, 
if we ignore the null-forms in a word,
then we have several right forms,
followed by zero or one 
combined form, followed by
several left forms.

We can describe each left form,
right form and combined form.
If a class is a left form, 
then it must be $``(''$ or $``(u)-(''$, 
where $u$ is some word.
Similarly,
if a class is a right form, 
then it must be $``)''$ or $``)-(u)''$.
If a class is a combined form, 
then it must be $``)-(''$ or $``)-(u)-(''$.
Based on these descriptions,
we define a way to change
one form into another. 
The transformations can be described
as follows.
\[
\begin{array}{c@{\quad\quad\longleftrightarrow\quad\quad}c@{\quad\quad\longleftrightarrow\quad\quad}c}
\textcolor{myblue}{)} & \textcolor{mypurple}{)-(} & \textcolor{myred}{(} \: , \\
\textcolor{myblue}{)-(u)} & \textcolor{mypurple}{)-(u)-(} & \textcolor{myred}{(u)-(}\: ,
\end{array}
\]
where $u$ is some word.
For instance, 
if we transform $``)'', ``)-()-(''$ and
$``)-(()())''$ into left forms, 
we get $``('', ``()-(''$ and
$``(()())-(''$.

\begin{defn}
Take $w \in \B_\word$.
If $w$ has no right forms, 
then $f$ sends it to $\emt$. 
Otherwise, if there is no combined form,
then $f$ transforms the 
last right form into a left form.
If there is a combined form,
then $f$ transforms the combined form 
into a left form 
and transforms the last right form into 
a combined form.
\end{defn}

\begin{exa}
\label{E: f on words}
We have 
$$
\textcolor{myblue}{)-()} \; \textcolor{myblue}{)} \;  \textcolor{myred}{()-(} \quad 
\xrightarrow{f}
\quad \textcolor{myblue}{)-()} \; \textcolor{myred}{(} \;  \textcolor{myred}{()-(}
$$
$$
\textcolor{myblue}{)-()} \; \textcolor{myblue}{)} \;  \textcolor{mypurple}{)-(} \; \textcolor{myred}{(} \quad 
\xrightarrow{f}
\quad \textcolor{myblue}{)-()} \;  \textcolor{mypurple}{)-(} \; \textcolor{myred}{(}\; \textcolor{myred}{(}.
$$
\end{exa}

The $e$ operator can be defined similarly.
\begin{defn}
Take $w \in \B_\word$.
If it has no left forms, 
then $e$ sends it to $\emt$. 
Otherwise, if there is no combined form,
then $e$ transforms the 
first left form into a right form.
If there is a combined form,
then $e$ transforms the combined form 
into a right form 
and transforms the first left form into 
a combined form.
\end{defn}

The $f$ and $e$ operators defined above
have ``square roots''. 
\begin{defn}
Define 
$$f', e': \B_\word 
\sqcup \{ \emt \}\rightarrow 
\B_\word \sqcup \{ \emt \}$$
based on the following cases:

\begin{enumerate}
\item $f'(\emt) 
= e'(\emt) = \emt$
\item Assume $w \in \B_\word$ 
has a combined form. 
The $f'$ will transform the combined form 
into a left form.
The $e'$ will transform the combined form 
into a right form.
\item Assume $w \in \B_\word$ has 
no combined form.
If there is a right form,
$f'$ will transform the last right form 
into a combined form.
Otherwise, $f'(w) = \emt$.
If there is a left form,
$e'$ will transform the last left form 
into a combined form.
Otherwise, $e'(w) = \emt$.
\end{enumerate}
\end{defn}

\begin{exa}
For example
$$
\textcolor{myblue}{)-()} \; \textcolor{myblue}{)} \;  \textcolor{myred}{()-(}
 \quad \xrightarrow{f'}  \quad 
\quad \textcolor{myblue}{)-()} \; 
\textcolor{mypurple}{)-(} \;  \textcolor{myred}{()-(}
\quad \xrightarrow{f'}  \quad 
\textcolor{myblue}{)-()} \; 
\textcolor{myred}{(} \;  \textcolor{myred}{()-(}
$$
\end{exa}

We have the following relations between
$f, e, f'$ and $e'$.
\begin{lemma}
Take $w_1, w_2$ from $B_\word$.
Then we have:
\begin{itemize}
\item $f'(w_1) = w_2$ if and only if $e'(w_2) = w_1$. 
\item $f(w_1) = w_2$ if and only if $f'(f'(w_1)) = w_2$.
\item $f(w_1) = w_2$ if and only if $e(w_2) = w_1$. 
\end{itemize}
\end{lemma}
\begin{proof}
Immediate from definitions. 
\end{proof}

\begin{rem}
\label{R: How f' works}
Assume $f'(w) \neq \emt$.
Then $f'$ must do 
one of the following to $w$: 
\begin{enumerate}
\item[$\bullet$] 
If $w$ has no combined form, $f'$ 
changes the last character of the last right form, 
which is $``)''$, into $``)-(''$.
Thus, $f'(w)$ has a combined form.
\item[$\bullet$] If $w$ has a combined form, 
$f'$ changes the first 3 characters of the combined form,
which are $``)-(''$, into $``(''$.
Thus, $f'(w)$ has no combined form.
\end{enumerate}

As a summary, $f'$ decreases the number of $``(''$ by 1
or increases the number of $``)''$ by ~1.
Consequently, $f$ decreases the number of $``(''$ by 1
and increases the number of $``)''$ by ~1.
\end{rem}

Finally, we are ready to define 
an abstract Kashiwara $\textrm{GL}_n$-crystal
on $\B_n$.
We start with $\varepsilon(\cdot)$
and $\varphi(\cdot)$.

\begin{defn}
Take $T \in \B_n$ and take $i \in [n-1]$.
Let $\varepsilon_i(T)$ (resp. $\varphi_i(T)$)
be the number of left forms (resp. right forms)
in the $i$-word of $T$.
\end{defn}

Now we would like to define the crystal operators 
$f_i$ and $e_i$ on $\B_n$.
First, we define their ``square roots'' 
$f_i'$ and $e_i'$:

\begin{defn}
Define $f_i'$ on $\B_n \cup \{\emt\}$.
First, $f_i'(\emt) = \emt$.
Now take $T \in \B_n$.
Apply $f'$ on the $i$-word of $T$.
Change $T$ accordingly and obtain $f_i'(T)$. 
More explicitly, 
based on Remark ~\ref{R: How f' works},
we may describe $f_i'$ as the following:
\begin{enumerate}
\item[$\bullet$] 
Assume $f'$ sends the $i$-word of $T$ to $\emt$.
Then $f_i'(T) = \emt$.
\item[$\bullet$] 
Assume $f'$ changes a $``)''$ 
into $``)-(''$.
Find the set in $T$ corresponding to this $``)''$ 
and add $i+1$ to this set.
\item[$\bullet$] 
Assume $f'$ changes a $``)-(''$ into $``(''$.
Find the set in $T$ corresponding to this $``)-(''$ 
and remove the $i$ in it.
\end{enumerate}
\end{defn}

\begin{exa}
\label{E: f_i'}
Consider $T$ in 
Example~\ref{E: i-word}.
The operator $f'$ would turn
its $2$-word $``())-())-(''$
into $``())-()(''$,
by changing the last $``)-(''$
into $``(''$.
Accordingly, 
we remove 2 from the set corresponding
to this $``)-(''$, 
which is the highest entry 
in column 4 of $T$.
Thus, we have
$$T = 
\raisebox{0.3cm}{
\begin{ytableau}
1 & 1 & 2 & 23\cr
2 & 23\cr
34 \cr 
\end{ytableau}}
\quad \xrightarrow{f_2'}  \quad 
\raisebox{0.3cm}{
\begin{ytableau}
1 & 1 & 2 & 3\cr
2 & 23\cr
34 \cr 
\end{ytableau}}
= f_2'(T).
$$

Now $f_2'(T)$
has 2-word $``())-()(''$.
The operator $f'$ would turn
this word into $``())-()-((''$
by changing the last $``)''$
into $``)-(''$.
Accordingly, 
we add 3 to the set
corresponding to this $``)''$,
which is the highest entry 
in column 3 of $f_2'(T)$.
Thus, we have
$$f_2'(T) = 
\raisebox{0.3cm}{
\begin{ytableau}
1 & 1 & 2 & 3\cr
2 & 23\cr
34 \cr 
\end{ytableau}}
\quad \xrightarrow{f_2'}  \quad 
\raisebox{0.3cm}{
\begin{ytableau}
1 & 1 & 23 & 3\cr
2 & 23\cr
34 \cr 
\end{ytableau}}
=f_2'(f_2'(T)).
$$
\end{exa}

\begin{lemma}
\label{L: f_i' well-defined}
The operator $f_i'$ is well-defined.
That is, for any $T \in \B_n$,
$f_i'(T)$ is a SVT or $\emt$. 
\end{lemma}
\begin{proof}
Assume $f_i'(T) \neq \emt$ 
and consider what $f'$ does
on the $i$-word of $T$.
\begin{itemize}
\item
Assume $f'$ changes 
a $``)''$ into a $``)-(''$.
Let $r$ denote this $``)''$.
In the $i$-word of $T$,
we know there is no combined form
and $r$ is the last character
of the last right form.
Let $S$ be the entry in $T$ 
corresponding this $r$.
The operator $f_i'$ will put
$i+1$ into $S$.
Let $S_\downarrow$ be the entry below $S$
and $S_\rightarrow$ be the entry on the 
right of $S$.
We need to check
(1) $S_\downarrow$, if exists, has no $i+1$; 
(2) $S_\rightarrow$, if exists, has no $i$.

\begin{itemize}
\item Assume (1) is false. 
In the $i$-word of $T$,
there is a $``(''$ immediately before $r$.
Then $r$ cannot be the last character of a 
right form. 
Contradiction.
\item Assume (2) is false.
Since there is no $i+1$ in $S_\downarrow$
if it exists, 
there are no $i+1$ below $S_\rightarrow$.
We know $S_\rightarrow$ corresponds to $``)-(''$
or $``)''$.
In either case, the $``)''$ is unpaired.
It must be part of a right form 
or a combined form.
However, there is no combined form
or right form after $r$.
Contradiction. 
\end{itemize}

\item 
Assume $f'$ changes a $``)-(''$ into $``(''$.
Then $f_i'$ removes $i$
from a set containing both $i$ and $i+1$. 
It will not cause any violation.
\end{itemize}
\end{proof}

\begin{defn}
Define $f_i : \B_n \rightarrow \B_n 
\sqcup \{ \emt \}$ 
as $f_i(T) = f_i'(f_i'(T))$.
Equivalently, 
we may define $f_i(T)$ as: 
Apply $f$ on the $i$-word of $T$
and change $T$ accordingly.
\end{defn}

\begin{exa}
\label{E: f_i}
Following Example~\ref{E: f_i'},
we have
$$T = 
\raisebox{0.3cm}{
\begin{ytableau}
1 & 1 & 2 & 23\cr
2 & 23\cr
34 \cr 
\end{ytableau}}
\quad \xrightarrow{f_2}  \quad 
\raisebox{0.3cm}{
\begin{ytableau}
1 & 1 & 23 & 3\cr
2 & 23\cr
34 \cr 
\end{ytableau}}
=f_2(T).
$$
\end{exa}

We can define $e_i'$ and $e_i$ 
on $\B_n$ similarly.

\begin{defn}
Define $e_i'$ on $\B_n \sqcup \{\emt\}$.
First, $e_i'(\emt) = \emt$.
Now take $T \in \B_n$.
Apply $e'$ on the $i$-word of $T$.
Change $T$ accordingly and obtain $e_i'(T)$. 
Then $e_i(T) := e_i'(e_i'(T))$.
\end{defn}

Similar to Lemma~\ref{L: f_i' well-defined},
we can show $e_i'$ is also well-defined. 

\begin{rem}
When we restrict $f_i, e_i, \varphi_i$ 
and $\varepsilon_i$ on a SSYT, 
we obtain the classical construction
described in \S\ref{S: Background}.
\end{rem}

\begin{lemma}
\label{L: Kashiwara on SVT}
$\B_n$, together with $f_i, e_i,
\varepsilon_i, \varphi_i$ and $\wt$,
is a seminormal abstract Kashiwara
$\textrm{GL}_n$-crystal.
\end{lemma}
\begin{proof}
First, we check the axioms
of an abstract Kashiwara
crystal.
\begin{itemize}
\item K1: Take $X, Y \in \B_n$.
Clearly, $e_i(X) = Y$ if and only if
$f_i(Y) = X$. 
Now assume this is the case.
The $i$-word of $Y$ has one more right form 
and one less left form
than the $i$-word of $X$,
so 
$\varphi_1(Y) 
= \varphi_1(X) + 1$ and 
$\varepsilon_1(Y) 
= \varepsilon_1(X) - 1$.

Finally, by Remark~\ref{R: How f' works},
the $i$-word of $Y$ has one more $``)''$
and one less $``(''$ than
the $i$-word of $X$.
In other words, 
$Y$ has one more $i$
and one less $i+1$ than $X$,
so $\wt(Y) = \wt(X) + v_i - v_{i+1}$.

\item 
K2: 
Take $X \in \B_n$.
The expression 
$\langle \wt(X), (v_i,-v_{i+1}) \rangle$
is the number of $i$ in $X$
minus the number of $i+1$ in $X$.
Thus, 
it is also the number of $``)''$ in $w$
minus the number of $``(''$ in $w$,
where $w$ is the $i$-word of $X$.

In each right form,
there is one more $``)''$ than $``(''$.
In each left form, 
there is one more $``(''$ than $``)''$.
In each combined form or null form,
the numbers of $``(''$ and $``)''$ 
are equal.
Thus, $\langle \wt(X), (v_i,-v_{i+1}) \rangle$
is also the number of right forms in $w$
minus the number of left forms in $w$,
which is~$\varphi_1(X) - \varepsilon_1(X)$. 
\end{itemize}

Next, we check it is seminormal.
Take $X \in \B_n$.
Each time we apply $e_i$,
the $i$-word of $X$ would lose one left form.
Thus, $e_1^{\varepsilon_i(X)}(X)$
has no left form.
We have 
$\varepsilon_i(X) 
= \max\{k: e_i^k(X) \neq \emt \}$.
The other equality can be proved similarly.

\end{proof}

\subsection{Double \boldmath{$i$}-strings}

In this subsection, 
we introduce and investigate double $i$-strings,
which can be viewed as analogues
of $i$-strings. 
Recall that an $i$-string in $\B_n$
is a sequence of SVT $T_0, \dots, T_k$
such that $e_i(T_0) = f_i(T_k) = \emt$ 
and $f_i(T_j) = T_{j+1}$
for $j = 0,1, \dots, k-1$.

\begin{exa}
The following are $2$-strings in $\B_3$.
\[
\begin{tikzcd}[ampersand replacement=\&]
\begin{ytableau}
1 & 2\cr
2 
\end{ytableau}\arrow[r, "2"]  \& 
\begin{ytableau}
1 & 3\cr
2 
\end{ytableau} \arrow[r, "2"]\& 
\begin{ytableau}
1 & 3\cr
3 
\end{ytableau}
\\
\begin{ytableau}
1 & 23\cr
2 
\end{ytableau} \arrow[r, "2"] \&
\begin{ytableau}
1 & 3\cr
23 
\end{ytableau}
\end{tikzcd}
\]
\end{exa}

Now we are ready to introduce 
the double $i$-string. 
 
\begin{defn}
Take $i \in [n-1]$.
A {\em double $i$-string} is a sequence
of $SVT$ $T_0, \dots, T_k \in \B_n$
satisfying:
\begin{enumerate}
\item[$\bullet$] $e_i'(T_0) = f_i'(T_k) = \emt$
\item[$\bullet$] $f_i'(T_j) = T_{j+1}$
for each $j \in \{0, 1, \dots, k-1 \}$.
\end{enumerate}
We say $T_0$ is the {\em source} of its double $i$-string.
Diagrammatically, we can represent the double $i$-string as:
\[
\begin{tikzcd}
T_0 \arrow[r,"i"] \arrow[d,dashrightarrow,"i"] & 
T_2 \arrow[r,"i"] \arrow[d,dashrightarrow,"i"] & 
T_4 \arrow[r,"i"] \arrow[d,dashrightarrow,"i"] &
\cdots \arrow[r,"i"] & 
T_{k-2} \arrow[r,"i"] \arrow[d,dashrightarrow,"i"] & 
T_k,
\\
T_1 \arrow[ur,dashrightarrow,"i"] \arrow[r,"i"] &
T_3 \arrow[ur,dashrightarrow,"i"] \arrow[r,"i"] & 
T_5 \arrow[r,"i"] & 
\cdots \arrow[r,"i"] & 
T_{k-1} \arrow[ur,dashrightarrow,"i"]
\end{tikzcd}
\]
where solid arrow represents $f_i$ 
and dash arrow represents $f_i'$.
\end{defn}

\begin{rem}
A double $i$-string can be viewed as a refinement
of the ``$i$-K-string'' in ~\cite{MPS}.
If we remove all dash arrows except the one
from $T_0$ to $T_1$,
we get an $i$-K-string.
\end{rem}

We make some basic observations about
a double $i$-string. 

\begin{lemma}
\label{L: Double i string}
Let $T_0, \dots, T_k$ be a double $i$-string.
Then we have the following.
\begin{enumerate}
\item $k$ is even.
\item If $k \geq 2$,
then this double $i$-string consists of
two $i$-strings: $T_0, T_2, \dots, T_k$
and $T_1, T_3, \dots, T_{k-1}$.
An element in the former $i$-string
has no combined form in its $i$-word.
An element in the latter $i$-string
has a combined form in its $i$-word.
\item $\wt(T_{2j+1}) = \wt(T_{2j}) + v_{i+1}$.
\item $\wt(T_{2j}) = \wt(T_{2j-1}) - v_i$.
\end{enumerate}
\end{lemma}

\begin{proof}
\begin{enumerate}
We know $T_0$ and $T_k$ have no combined forms
in their $i$-word.
By the definition $f_i'$,
the $i$-word of $T_{j+1}$ has a combined form 
if and only if
the $i$-word of $T_j$ has no combined form.
Thus, we have (1) and (2).
The other two statements follow
from Remark~\ref{R: How f' works}.
\end{enumerate}
\end{proof}

\begin{exa}
\label{E: Double 2-string}
Notice that the following is a double $2$-string
in $\B_3$:
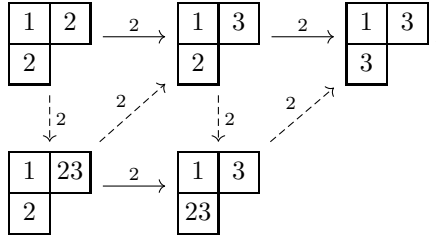
\[
\begin{tikzcd}[ampersand replacement=\&]
\begin{ytableau}
1 & 2\cr
2 
\end{ytableau}\arrow[r, "2"]  \arrow[d,dashrightarrow,"2"]\& 
\begin{ytableau}
1 & 3\cr
2 
\end{ytableau} \arrow[r, "2"]\arrow[d,dashrightarrow,"2"]\& 
\begin{ytableau}
1 & 3\cr
3 
\end{ytableau} \: .
\\
\begin{ytableau}
1 & 23\cr
2 
\end{ytableau} \arrow[ur,dashrightarrow,"2"] \arrow[r, "2"] \&
\begin{ytableau}
1 & 3\cr
23 
\end{ytableau}\arrow[ur,dashrightarrow,"2"]
\end{tikzcd}
\]

This double $2$-string consists of two $2$-strings 
that appear in the previous example.
Observe that the three SVT in the first row
do not have combined forms in their $2$-words,
while the two SVT on the second row have.

\end{exa}

\subsection{Double \boldmath{$i$}-string and the right key}
This subsection investigates how
the right key is changed in a double $i$-string.
More explicitly, we prove:
\begin{lemma}
\label{L: double i-string and right key}
Let $T_0, T_1, \dots, T_{2k}$ be 
a double $i$-string in $\B_n$.
Assume $K_+(T_0) = \key(\alpha)$.
Then $\alpha_i \geq \alpha_{i+1}$, 
and there are two possibilities:
\begin{enumerate}
\item[$\bullet$] 
$K_+(T_1) = \dots = K_+(T_k) = \key(\alpha)$, or 
\item[$\bullet$] 
$K_+(T_1) = \dots = K_+(T_k) = \key(s_i \alpha)$.
\end{enumerate}
\end{lemma}

\begin{exa}
Let $T_0, \dots, T_4$ 
be the double $2$-string of $\B_3$ 
in Example~\ref{E: Double 2-string}.
We have $K_+(T_0) = \key(\alpha)$
and $K_+(T_1) = \cdots = K_+(T_4) = \key(s_i \alpha)$,
where $\alpha = (1,2,0)$

\end{exa}

The main goal of this subsection is to prove
Lemma ~\ref{L: double i-string and right key}.
by investigating how $f_i'$ and $e_i'$
change the right key of a SVT whose $i$-string
has a combined form.
First, we need a few lemmas
about the $\star$ operator.
\begin{lemma}
\label{L: Has i no i+1}
Let $S$ be a finite subset of $\mathbb{Z}$.
Pick $i \in \mathbb{Z}$
and assume $w$ is a word of $\mathbb{Z}$
with no $i+1$.
Then if $S \star iw$ contains $i+1$, 
it must also contain $i$.
\end{lemma}
\begin{proof}
If $i+1 \notin S \star i$,
then $i+1 \notin S\star iw$ since
$w$ has no $i+1$.
We are done in this case.
Otherwise, $i, i+1 \in S \star i$.
When $w$ acts on $S\star i$,
to change the $i$,
it first needs to bump the $i+1$.
Thus, $i$ remains in $S \star iw$
if it contains $i+1$.
\end{proof}

\begin{defn}
Let $S$ a finite subset of $\mathbb{Z}$.
We define the set $\partial_i S$ 
according to the following cases
\begin{itemize}
\item If $i, i+1 \in S$, 
then $\partial_i S = S$.
\item If $i \not\in S$ and $i+1 \not\in S_1$,
then $\partial_i S = S$.
\item If $i \in S$ and $i+1 \not\in S_1$,
then $\partial_i S = S - \{i\} \sqcup \{i+1\}$.
\item If $i \notin S$ and $i+1 \in S_1$,
then $\partial_i S$ is undefined.
\end{itemize}
\end{defn}

\begin{lemma}
\label{L: Star changes i into i+1}
Let $S_1$ be a set such that
$\partial_i(S_1)$ is defined.
Let $S_2 = \partial_i S_1$.
Then we have the following. 
\begin{itemize}
\item For any $x \neq i$ or $i+1$, 
the set $S_2 \star x$ is $S_1 \star x $
or $\partial_i (S_1 \star x)$;
\item $S_2 \star (i+1) = S_1 \star (i+1)$.
\end{itemize}
\end{lemma}
\begin{proof}
If $S_1 = S_2$, then clearly
$S_2 \star x = S_1 \star x$
and $S_2 \star (i+1) = S_1 \star (i+1)$.
Now assume $S_1 \neq S_2$
(i.e. $i \in S$ and $i+1 \notin S$).
We know $S_2$ is obtained by 
changing the $i$ in $S_1$ into $i+1$.
We check the two statements.
\begin{itemize}
\item If $x$ bumps some $y \neq i$ in $S_1$
or adds itself to $S_1$,
then $x$ would do the same in $S_2$, so
$S_2 \star x = \partial_i (S_1 \star x)$.
Now if $x$ bumps $i$ in $S_1$,
then it would bump $i+1$ in $S_2$,
so $S_2 \star x = S_1 \star x$.
\item The $i+1$ must bump $i$ in $S_1$
and $i+1$ in $S_2$, so
$S_2 \star (i+1) = S_1 \star (i+1)$.
\end{itemize}
\end{proof}

\begin{lemma}
\label{L: f' and right key}
For $T \in B_n$
and $i \in [i-1]$,
$K_+(f_i'(T)) = K_+(T)$
if $T$ has a combined form
in its $i$-word.
\end{lemma}
\begin{proof}
Assume $f_i'$ removes $i$ from
the entry $S$, 
which is in column $c$ of $T$.
Then clearly 
$K_+(T)$ and $K_+(f_i'(T))$
must agree on column $j$ 
if $j > c$.
We only need to worry about
column $j$ of 
$K_+(T)$ and $K_+(f_i'(T))$
for $j \leq c$.
Let $T_{\geq j}$ be the SVT 
obtained by removing the first 
$j-1$ columns of $T$.
Let $u = word(T_{\geq j})$.
Recall that column $j$ of $K_+(T)$
is $\emptyset \star u$.

We may write $u$ as $u_1 \:\:\: i \:\:\: i+1 \:\:\: u_2$,
where the $i$ and $i+1$ correspond to the 
$i$ and $i+1$ in $S$.
Then column $j$ of $K_+(f_i'(T))$
is $\emptyset \star u_1 \:\: (i+1) \:\: u_2$.
Thus, it remains to prove:
\begin{align}
\label{Equ: f_i'}
(\emptyset \star u_1) \star \:\: i \:\: (i+1)= 
(\emptyset \star u_1) \star \:\: (i+1)   
\end{align}

Now we consider the $i$-word of $T$. 
The combined form must follow
a right form or a null-form 
or nothing.
Thus, the character before the 
combined form must be $``)''$ 
or nothing.
In other words,
$u_1$ has two possibilities:
has neither $i$ nor $i+1$,
or has the form $u_1^1 \:\: i \:\: u_1^2$,
where $u_1^2$ has no $i+1$.
By Lemma ~\ref{L: Has i no i+1}, 
we have either $i+1 \not\in \emptyset \star u_1$ or
$i, i+1 \in \emptyset \star u_1$.
Now we study these two cases.
\begin{enumerate}
\item Assume we have the former case. 
If we let $i$ act on $\emptyset \star u_1$,
it will change a number into $i$,
or add itself to it. 
Then if we let $i+1$ act on the result, 
it will replace the $i$ by $i+1$,
which is the same as $(\emptyset \star u_1) \star (i+1)$.
\item Assume we have the latter case.
Action of $i$ or $i+1$ on $\emptyset \star u_1$
will not do anything. 
Both sides of~(\ref{Equ: f_i'})
must agree with $\emptyset \star u_1$.
\end{enumerate}
\end{proof}

Similarly, for $e_i'$, we have: 

\begin{lemma}
\label{L: e' and right key}
Take $T \in B_n$ and $i \in [i-1]$.
Assume $T$ has a combined form
in its $i$-string.
Assume $K_+(T) = \key(\alpha)$.
If $T$ also has a left form, 
then $K_+(e_i'(T)) = \key(\alpha)$.
If $T$ has no left form,
then $K_+(e_i'(T)) = \key(\alpha)$ 
or $\key(s_i \alpha)$.
\end{lemma}

\begin{proof}
Assume $e_i'$ removes $i+1$ from
the entry $S$, 
which is in column $c$ of $T$.
Then clearly column $j$ of 
$K_+(T)$ and $K_+(e_i'(T))$
must agree if $j > c$.
We only need to worry about
column $j$ of 
$K_+(T)$ and $K_+(f_i'(T))$
for $j \leq c$.
Let $T_{\geq j}$ be the SVT 
obtained by removing the first 
$j-1$ columns of $T$.
Let $u = word(T_{\geq j})$.
Recall that column $j$ of $K_+(T)$
is $\emptyset \star u$.

We may break $u$ into 
$u_1 \:\:\: i \:\:\: (i+1) \:\:\: u_2$,
where the $i$ and $i+1$ correspond to the 
$i$ and $i+1$ in $S$.
Then column $j$ of $K_+(e_i'(T))$
is $\emptyset \star u_1 \:\: i \:\: u_2$.
Thus, it remains to compare:
$$
(\emptyset \star u_1) \star \:\: i \:\: (i+1) \:\: u_2 
\text{ and }
(\emptyset \star u_1) \star \:\: i \:\: u_2.
$$

Let $S_1 = (\emptyset \star u_1) \star i$ 
and $S_2 = (\emptyset \star u_1) \star i \: (i+1)$.
Clearly, $i \in S_1$.
If $i+1 \in S_1$, then $i, i+1 \in S_1$ and $S_1 = S_2$.
If $i+1 \not\in S_1$, then 
$S_2 = S_1 - \{i\} \sqcup \{i+1\}$.
In either case, 
we have $S_2 = \partial_i S_1$.

Now we think about the $i$-word of $T$. 
The combined form must be followed by
a left form or a null-form 
or nothing.
Thus, the character after the 
combined form must be $``(''$ 
or nothing.
In other words,
we have two cases.
\begin{itemize}
\item Case 1: The word $u_2$ can be written as
$u_2^1 \:\: (i+1) \:\: u_2^2$.
By Lemma ~\ref{L: Star changes i into i+1},
$S_2 \star u_2^1 = \partial_i(S_1 \star u_2^1)$.
Then $S_2 \star u_2^1 \:\: (i+1)= S_1 \star u_2^1 \:\: (i+1)$,
so $S_2 \star u_2 = S_1 \star u_2$.
\item Case 2: The word $u_2$ has no $i$ or $i+1$.
By Lemma ~\ref{L: Star changes i into i+1},
$S_2 \star u_2 = \partial_i(S_1 \star u_2)$.
\end{itemize}
The second case is possible
only when the $i$-word of $T$ has no left form.
This is exactly what we need to prove.
\end{proof}

Now we are ready to prove 
Lemma ~\ref{L: double i-string and right key}.
\begin{proof}
[Proof of Lemma 
~\ref{L: double i-string and right key}]
First, we consider $T_0$.
Since it has neither combined form nor left form,
its last character in the $i$-string, if exists, 
must be $``)''$.
Thus, columns of $K_+(T_0)$ will be
$\emptyset \star u_1 \: i \: u_2$ or $\emptyset \star u$,
where $u_2$ and $u$ have no $i$ or $i+1$.
By Lemma ~\ref{L: Has i no i+1},
if a column of $K_+(T_0)$ has $i+1$,
it must also have $i$.
Thus, $\alpha_i \geq \alpha_{i+1}$.

Now by Lemma ~\ref{L: f' and right key},
we know $K_+(T_{2j-1}) = K_+(T_{2j})$
where $j \in [k]$.
By Lemma ~\ref{L: e' and right key}
we know $K_+(T_{2j}) = K_+(T_{2j+1})$
where $j \in [k]$.
Thus, $T_1, \dots, T_{2k}$ all have the same 
right key.

Finally, notice that $T_1$ is the 
source of its $i$-string,
so it has no left form. 
By Lemma ~\ref{L: e' and right key} again, 
$K_+(T_1) = \key(\alpha)$ or $\key(s_i \alpha)$,
where $\alpha = K_+(T_0)$.
\end{proof}

\begin{cor}
\label{C: Must be source}
Let $T$ be a SVT.
If $f_i(T) \neq \emt$ and 
$K_+(T) \neq K_+(f_i(T))$,
then $T$ must be the source of its double $i$-string. 
\end{cor}

\subsection{Proof of Theorem ~\ref{T: Main}}

In this subsection, we derive a few lemmas
and then use them to prove Theorem ~\ref{T: Main}. 
First, we describe a well-known result
that is implicit in~\cite{K}.
It states that the generating function of each
$i$-string behaves nicely under $\pi_i$.
For the sake of completeness, 
we provide a brief proof. 

\begin{lemma}
\label{L: Sum over an i-string}
For each $i$-string $T_0, \dots, T_{k}$,
we have
$$
\pi_i(x^{\wt(T_0)})
= \sum_{j=0}^{k} x^{\wt(T_j)}.
$$
\end{lemma}
\begin{proof}
Write $x^{\wt(T_0)}$ as $m x_i^{a} x_{i+1}^{b}$,
where $m$ is a monomial with no $x_i$ or $x_{i+1}$.
By Lemma~\ref{L: string weight},
$x^{\wt(T_k)} = m x_i^{b} x_{i+1}^{a}$.
Thus, $k = b - a$.
Finally, we have 
\begin{equation*}
\begin{split}
\pi_i( x^{\wt(T_0)} ) 
= & \:m \pi_i(x_i^{a} x_{i+1}^{b} ) \\
= & \:m \sum_{j=0}^{b-a} x_i^{a-j} x_{i+1}^{b+j} \\
= & \:\sum_{j=0}^{k} x^{\wt(T_{j})}.\\
\end{split}
\end{equation*}
\end{proof}

As mentioned earlier, 
double $i$-string can be viewed
as a refinement of $i$-K-string in ~\cite{MPS}.
Authors of ~\cite{MPS} knew that
the generating function of an $i$-K-string 
behaves nicely under $\pib_i$:
Applying $\pib_i$ on the weight of the source
yields the generating function of a whole $i$-K-string. 
This property is also satisfied by 
double $i$-strings.
The following is implicit in~\cite[Theorem 7.5]{MPS}.

\begin{lemma}
\label{L: Sum over a double i-string}
For each double $i$-string $T_0, \dots, T_{2k}$,
we have
$$
\pib_i( x^{\wt(T_0)}\beta^{\ex(T_0)} )
= \sum_{j=0}^{2k} x^{\wt(T_j)} \beta^{\ex(T_j)},
$$
$$
\pib_i( \sum_{j=0}^{2k} x^{\wt(T_j)} \beta^{\ex(T_j)} )
= \sum_{j=0}^{2k} x^{\wt(T_j)} \beta^{\ex(T_j)}.
$$
\end{lemma}

\begin{proof}
First, we establish the first equation 
using the argument in~\cite{MPS}.
Notice that 
$$
\pib_i(f) = \pi_i(f + \beta x_{i+1}f).
$$
Thus, its left hand side becomes
$$
\pi_i(x^{\wt(T_0)}\beta^{\ex(T_0)} 
+ \beta x_{i+1}x^{\wt(T_0)}\beta^{\ex(T_0)}).
$$
Notice that $x^\wt(T_1) = x^\wt(T_0) x_{i+1}$
and $\ex(T_1) = \ex(T_0) + 1$.
We can further simplify the left hand side into
\begin{align*}
 & \pi_i(x^{\wt(T_0)}\beta^{\ex(T_0)} 
+ x^{\wt(T_1)}\beta^{\ex(T_1)})\\
= & \pi_i(x^{\wt(T_0)}\beta^{\ex(T_0)}) 
+ \pi_i(x^{\wt(T_1)}\beta^{\ex(T_1)}).  
\end{align*}
Then the first equation is established by 
Lemma ~\ref{L: Sum over an i-string}.

For the second equation, notice that
$\sum_{j=0}^{2k} x^{\wt(T_j)} \beta^{\ex(T_j)}$
is symmetric in $x_i$ and $x_{i+1}$.
Then the equation is established by the fact: 
$\pib_i(f) = f$ if $s_i(f) = f$.

\end{proof}

Next, we describe $\SVT(\alpha)$ in terms of double $i$-strings.
\begin{lemma}
Take any weak composition $\alpha$.
For each double $i$-string $T_0, \dots, T_{2k}$,
if $T_i \in \SVT(\alpha)$ with $i > 0$,
then $T_0, \dots, T_{2k} \in \SVT(\alpha)$.
\end{lemma}
\begin{proof}
We know $K_+(T_i) \leq \key(\alpha)$.
Since $T_1, \dots, T_{2k}$ all have the same right key,
they are all in $\SVT(\alpha)$.
By Lemma ~\ref{L: double i-string and right key},
$K_+(T_0) \leq K_+(T_i)$, so $T_0 \in \SVT(\alpha)$.
\end{proof}

The following is analogous to ~\cite[Proposition ~3.3.5]{K}.
\begin{cor}
\label{C: All, none, or source}
Take any weak composition $\alpha$.
For each double $i$-string $S = \{T_0, \dots, T_{2k}\}$,
then $\SVT(\alpha) \cap S$ is $S$, $\emptyset$, 
or $\{T_0\}$.
\end{cor}

\begin{lemma}
\label{L: action s_i on crystal1}
Let $\alpha$ be a weak composition such that $\alpha_i > \alpha_{i+1}$.
We can decompose $\SVT(s_i \alpha)$ into a disjoint union of double $i$-strings. 
For each of the double $i$-string in $\SVT(s_i \alpha)$, 
$\SVT(\alpha)$ either contains its source or all of it. 
\end{lemma}

\begin{exa}
\label{E: 3 double 2-strings}
When $\alpha = (1,2,0)$, 
the set $\SVT(s_2 \alpha)$ is a disjoint union of 
three double $2$-strings. 
Besides the double 2-string 
in Example~\ref{E: Double 2-string},
it also contains
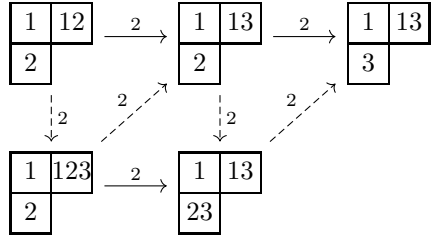
\[
\begin{tikzcd}[ampersand replacement=\&]
\begin{ytableau}
1 & 12\cr
2 
\end{ytableau}\arrow[r, "2"]  \arrow[d,dashrightarrow,"2"]\& 
\begin{ytableau}
1 & 13\cr
2 
\end{ytableau} \arrow[r, "2"]\arrow[d,dashrightarrow,"2"]\& 
\begin{ytableau}
1 & 13\cr
3 
\end{ytableau}
\\
\begin{ytableau}
1 & 123\cr
2 
\end{ytableau} \arrow[ur,dashrightarrow,"2"] \arrow[r, "2"] \&
\begin{ytableau}
1 & 13\cr
23 
\end{ytableau}\arrow[ur,dashrightarrow,"2"]
\end{tikzcd}
\]
and 
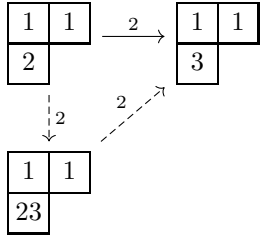
\[
\begin{tikzcd}[ampersand replacement=\&]
\begin{ytableau}
1 & 1\cr
2 
\end{ytableau}\arrow[r, "2"]  \arrow[d,dashrightarrow,"2"]\& 
\begin{ytableau}
1 & 1\cr
3 
\end{ytableau} \:.
\\
\begin{ytableau}
1 & 1\cr
23 
\end{ytableau} \arrow[ur,dashrightarrow,"2"] 
\end{tikzcd}
\]

For each of these three double 2-strings,
the set $\SVT(\alpha)$ only contains 
the source. 

\end{exa}

\begin{proof}
Let $T_0, \dots, T_{2k}$ be a double $i$-string
that intersects with $\SVT(s_i\alpha)$.
Corollary ~\ref{C: All, none, or source} implies
$T_0 \in \SVT(s_i\alpha)$.
Let $\gamma = wt(K_+(T_0))$,
then $\key(\gamma) \leq \key(s_i\alpha)$.
We know each SVT in this double $i$-string
has right key $\key(\gamma)$ or $\key(s_i(\gamma))$.
Since $\alpha_i > \alpha_{i+1}$,
$\key(s_i\gamma) \leq \key(s_i\alpha)$.
Thus, the whole double $i$-string
is in $\SVT(s_i \alpha)$.

Lemma ~\ref{L: double i-string and right key} 
implies that $\gamma_i \geq \gamma_{i+1}$,
so $\key(\gamma) \leq \key(\alpha)$.
We have $T_0 \in \SVT(\alpha)$.
By Corollary ~\ref{C: All, none, or source}, 
$\SVT(\alpha)$ either contains $T_0$
or the whole double $i$-string. 
\end{proof}

Now we are ready to prove our main result:
\begin{proof}[Proof of Theorem ~\ref{T: Main}]
We only need to check $\sum_{T \in \SVT(\alpha)} x^{\wt(T)} \beta^{\ex(T)}$
satisfies the recursive definition of $\fLb_\alpha$.
In other words,
we need to prove:
\begin{enumerate}
\item[$\bullet$] 
If $\alpha$ is a partition, then
$$
\sum_{T \in \SVT(\alpha)} x^{\wt(T)} \beta^{\ex(T)} = x^\alpha.
$$ 
\item[$\bullet$] If $\alpha_i > \alpha_{i+1}$, then
\begin{equation}
\label{E: Main2}
 \pib_i (\sum_{T \in \SVT(\alpha)} x^{\wt(T)} \beta^{\ex(T)})
= \sum_{T \in \SVT(s_i\alpha)} x^{\wt(T)} \beta^{\ex(T)}
\end{equation}
\end{enumerate}
The first statement is immediate.
For the second one, 
we break $\SVT(\alpha)$ into $A \sqcup B$.
The set $A$ consists of all $T$ whose whole double $i$-string 
is in $\SVT(\alpha)$.
The set $B$ contains all $T \in \SVT(\alpha)$ such that 
part of its double $i$-string is not in $\SVT(\alpha)$.
Let $\overline{B}$ be the union of double $i$-strings
who intersect with $B$.
By Lemma ~\ref{L: action s_i on crystal1}, 
elements in $B$ are sources of double $i$-string and 
$\SVT(s_i \alpha) = A \sqcup \overline{B}$.
Now by Lemma ~\ref{L: Sum over a double i-string},
$$
\pib_i(\sum_{T \in A} x^{\wt(T)} \beta^{\ex(T)}) 
=  \sum_{T \in A} x^{\wt(T)} \beta^{\ex(T)},
$$
$$
\pib_i(\sum_{T \in B} x^{\wt(T)} \beta^{\ex(T)}) 
=  \sum_{T \in \overline{B}} x^{\wt(T)} \beta^{\ex(T)}.
$$

Equation ~\ref{E: Main2} is obtained
by summing up the two equations above. 

\end{proof}

\section{K-theory crystal}
\label{S: K-crystal}
In this section, 
we describe some similarities between our
abstract Kashiwara crystal and the Demazure crystal.
Then we explain why our crystal can be viewed 
as an answer to~\cite[Open Problem 7.1]{MPS}.

Similar to the $\F_i$ defined in \S\ref{S: Background},
we define $\F_i' S$ as 
$\{ (f_i')^j (T): T \in S, j \geq 0\} 
- \{ \emt\}$,
where $S \subseteq \B_n$.
Then we have an analogue of 
Theorem~\ref{T: obtain SSYT by f_i}:
\begin{thm}
\label{T: obtain SVT by f_i'}
Let $\alpha$ be a weak composition
such that $\alpha^+ = \lambda$
and $\alpha_i = 0$ for $i > n$.
We can write $\alpha$ as 
$s_{i_1}\dots s_{i_k} \lambda$,
where $k$ is minimal.
Then we have
\begin{equation*}
\SVT(\alpha)
= \F_{i_1}' \dots \F_{i_k}' \{ u_\lambda\}.
\end{equation*}
Here, $u_\lambda$ is the SSYT with shape $\lambda$
such that its $r^{th}$ row only has $r$.
\end{thm}
\begin{proof}
Prove by induction on $k$.
The base case is when $k = 0$ and $\alpha = \lambda$.
The equation becomes $\SVT(\alpha) = \{ u_\lambda\}$,
which is immediate.

The inductive step is to prove the following.
\begin{equation}
\label{E: apply f_i' on a set}
\SVT(s_i \alpha) 
= \F_i' \: \SVT(\alpha),
\end{equation}
where $\alpha_i > \alpha_{i+1}$.
We prove each side of 
the equation contains the other side. 
\begin{enumerate}
\item[$\bullet$] 
Take $T \in \SVT(\alpha)$.
By Lemma ~\ref{L: action s_i on crystal1},
the double $i$-string of $T$ is completely in $\SVT(s_i\alpha)$.
Thus, $(f_i')^j (T)$ is in $\SVT(s_i\alpha)$ if it is not $\emt$.
\item[$\bullet$] 
Suppose $T \in \SVT(s_i \alpha)$.
By Lemma ~\ref{L: action s_i on crystal1},
the source of its double $i$-string is in $\SVT(\alpha)$.
Thus, $T$ is in the right hand side.
\end{enumerate}
\end{proof}

If we slightly rephrase this theorem,
we get the following statements, 
which correspond to axioms K1 and K2 in section 7 of ~\cite{MPS}.
\begin{cor}
Let $\alpha$ be a weak composition
such that $\alpha^+ = \lambda$
and $\alpha_i = 0$ for $i > n$.
\begin{enumerate}
\item For each $T \in SVT(\alpha)$,
we can obtain $T$ by applying $f_i'$ on $u_\lambda$.
\item 
If $\alpha = s_{i_1}\dots s_{i_k} \lambda$
where $k$ is minimized,
then 
$$
\SVT(\alpha) \sqcup \{ \emt\} 
= \{ f_{i_1}'^{j_1} \dots f_{i_k}'^{j_k} u_\lambda: j_1, \dots, j_k \geq 0\}
$$
\end{enumerate}
\end{cor}

The third axiom in ~\cite{MPS} corresponds to our main result. 
Thus, we claim our construction is an answer to ~\cite[Open Problem 7.1]{MPS}
in the context of abstract Kashiwara crystals.

\section{Lascoux atoms}
\label{S: Atom}
In this section, 
we extend our rule to another set of polynomials
called Lascoux atoms. 
\begin{defn}
Following~\cite{Las01},
define the operator $\opib_i$ 
on $\mathbb{Z}[\beta][x_1, x_2, \dots]$ by 
$$
\opib_i(f) := \pib_i(f) - f.
$$
\end{defn}

\begin{defn}
Let $\alpha$ be a weak composition.
Similar to Lascoux polynomials,
a {\em Lascoux atom} $\ofLb_\alpha$
is defined as
$$
\ofLb_\alpha = \begin{cases}
x^\alpha & \text{if $\alpha$ is a partition} \\
\opib_i \ofLb_{s_i\alpha} &\text{if $\alpha_i<\alpha_{i+1}$.}
\end{cases}
$$
\end{defn}

The relationship between Lascoux polynomials
and Lascoux atoms can be described as follows.
\begin{defn}
Let $\alpha$ be a weak composition.
Following~\cite{Mon},
define $w(\alpha)$ as the shortest permutation
such that $\alpha = w(\alpha) \alpha^+$.
\end{defn}

\begin{lemma} [{\cite[Theorem 5.1]{Mon}}]
\label{L: Sum of atoms}
Let $\alpha$ be a weak composition.
Then 
$$
\fLb_\alpha = \sum_{\gamma} \ofLb_\gamma,
$$
where the sum is over all weak composition
$\gamma$ such that $\gamma^+ = \alpha^+$
and $w(\gamma) \leq w(\alpha)$ 
in Bruhat order.
\end{lemma}

It is well-known that the condition 
on $\gamma$ is the previous lemma
can be phrased as a condition 
on $\key(\gamma)$.
\begin{lemma}[{\cite[Equation (2.13)]{LS1}}]
\label{L: Bruhat order}
Let $\alpha$ and $\gamma$ be weak compositions
such that $\gamma^+ = \alpha^+$.
The following are equivalent.
\begin{itemize}
\item $w(\gamma) \leq w(\alpha)$ in Bruhat order.
\item $\key(\gamma) \leq \key(\alpha)$ entry-wise. 
\end{itemize}
\end{lemma}

Finally, we are ready to extend our rule
to Lascoux atoms. 
\begin{defn}
Let $\overline{\SVT}(\alpha)$
be the set of all SVT $T$ such that
$K_+(T) = \key(\alpha)$.
\end{defn}

\begin{rem}
By Lemma~\ref{L: Bruhat order},
we have 
$$
\SVT(\alpha) = \bigsqcup_{\gamma} \overline{\SVT}(\gamma),
$$
where $\gamma$ is any weak composition
such that $\gamma^+ = \alpha^+$ 
and $w(\gamma) \leq w(\alpha)$.
\end{rem}

\begin{cor}
We have
$$
\ofLb_\alpha = \sum_{T \in \overline{\SVT}(\alpha)}
x^{\wt(T)}\beta^{\ex(T)}.
$$
\end{cor}
\begin{proof}
Prove by induction on the Bruhat order.
If $\alpha$ is a partition,
then our result is immediate.

Now assume this rule holds for
all $\gamma$ such that 
$\gamma^+ = \alpha^+$ 
and $w(\gamma) < w(\alpha)$.
By Lemma~\ref{L: Sum of atoms},
we have
$$
\ofLb_\alpha = \fLb_\alpha 
- \sum_{\gamma}\ofLb_\gamma ,
$$
where the sum is over all $\gamma \neq \alpha$
such that $\gamma^+ = \alpha^+$ 
and $w(\gamma) \leq w(\alpha)$.
By our main result and the inductive hypothesis,
we have
$$
\ofLb_\alpha = \sum_{T \in \SVT(\alpha)}x^{\wt(T)}\beta^{\ex(T)} 
- \sum_{\gamma}\sum_{T \in \overline{\SVT}(\gamma)}
x^{\wt(T)}\beta^{\ex(T)}.
$$
Then the inductive step is finished 
by the remark above.
\end{proof}

\section{Future directions}
In this section, we introduce a few 
problems related to our main result.

\subsection{Finding a bijective proof 
of Theorem ~\ref{T: Main}}

As mentioned in \S\ref{S: Introduction},
there exist various combinatorial formulas
for Lascoux polynomials.
We would like to describe one of them.

A {\em reverse semistandard 
Young tableau} (RSSYT)
is a filling of a Young diagram with $\mathbb{Z}_{>0}$,
such that each column is strictly decreasing
and each row is weakly decreasing.
A {\em reverse set-valued tableau} (RSVT)
is a filling of Young is a filling of a Young diagram with non-empty subsets of $\mathbb{Z}_{>0}$,
such that no matter how we pick we number 
in each entry,
the resulting tableau is a RSSYT.
We may define $\wt(\cdot)$ and $\ex(\cdot)$
for RSVT analogously.
We also define a map $L(\cdot)$ from RSVT
to RSSYT.
$L(\cdot)$ picks the largest number in each 
entry.

A {\em reverse key} is a RSSYT,
where each number in column $j$ is also
in column $j-1$.
Clearly, reverse keys are in bijection
with weak compositions. 
We let $\key^R$ be the map that sends
a weak composition to its corresponding
reverse key.
Each RSSYT $T$ is associated with a reverse key 
called its left key, denoted by $K_-(T)$.

There is a weight-preserving map from SSYT
to RSSYT called reverse complement ~\cite{LS1}.
It is anti-rectification, followed by $180^o$
rotation.
Moreover, if $T$ is a SSYT 
with right key $\key(\alpha)$,
then the left key of $T$'s image is  $\key^R(\alpha)$.
Under this bijection,
we may transform the SSYT rule for Demazure character
~(\ref{E: SSYT rule of Demazure})
into a RSSYT rule:

\begin{equation}
\label{E: RSSYT rule of Demazure}
\kappa_{\alpha}= 
\sum_{K_-(T) \leq \key^R(\alpha)} x^{\wt(T)},
\end{equation}
where $T$ is a RSSYT.

This rule is generalized to Lascoux polynomials
by ~\cite{BSW, SY}:
\begin{equation}
\label{E: RSVT rule of Lascoux}
\fLb_{\alpha}= 
\sum_{K_-(L(T)) \leq \key^R(\alpha)} 
x^{\wt(T)} \beta^{\ex(T)},
\end{equation}
where $T$ is a RSVT.

One can prove Theorem ~\ref{T: Main} by 
building an appropriate bijection between
the SVT appeared in ~(\ref{E: Main})
and the RSVT in ~(\ref{E: RSVT rule of Lascoux}). 
More explicitly, 
we may describe the problem as follows.
\begin{prob}
Find a map $\Phi$ that sends a SSYT to a RSSYT,
satisfying:
\begin{enumerate}
\item $\Phi$ preserves $\wt(\cdot)$ and $\ex(\cdot)$.
\item If $T$ is a RSVT with $K_+(T) = \key(\alpha)$,
then $L(\Phi(T))$ has left key $\key^R(\alpha)$.
\end{enumerate}
\end{prob}

\subsection{Tableau complexes}

As mentioned in \S\ref{S: Introduction}, 
we can view $\SVT(\alpha)$ 
as a sub-complex of the Young tableau complex.
Knutson, Miller and Yong ~\cite{KMY} showed
that the Young tableau complex is homeomorphic to 
a ball or a sphere.

\begin{prob}
Determine whether the sub-complex $\SVT(\alpha)$
is also homeomorphic to a ball or a sphere.
\end{prob}

\section{Acknowledgments}
We are especially grateful to 
Travis Scrimshaw for carefully reading an earlier
version of this paper 
and giving many useful comments.
We also thank two anonymous referees for
thorough reports. 
We are grateful to Jianping Pan, Brendon Rhoades, Mark Shimozono, and Alexander Yong for helpful conversations. 

\bibliographystyle{alpha}
\bibliography{output.tex}

\end{document}